\documentclass[11pt]{amsart}
\usepackage{amsmath, amssymb, amscd, amsthm, mathtools, amsfonts, mathrsfs, tikz-cd, bbm, enumitem}
\usepackage{lipsum}
\usepackage{adjustbox}
\usepackage[mathscr]{euscript}
\usepackage{hyperref}
\usepackage[utf8]{inputenc}
\usepackage[english]{babel}
\usepackage{stmaryrd}
\usepackage{todonotes}
\usepackage[T1]{fontenc}
\usepackage{float}
\usepackage{arydshln}

\renewcommand{\int}{\smallint\!}

\newcommand{\op}{\mathrm{op}}
\DeclareMathOperator*{\colim}{colim}

\newcommand{\fp}{\mathrm{fp}}

\newcommand{\pf}{\mathrm{pf}}
\newcommand{\inv}{^{-1}}
\let\bar\overline

\DeclareMathOperator{\id}{id}
\newcommand{\cdh}{\mathrm{cdh}}

\newcommand{\red}{\mathrm{red}}

\newcommand{\pfp}{\mathrm{pfp}}

\newcommand{\pcdh}{\mathrm{pcdh}}
\newcommand{\idh}{\mathrm{idh}}
\newcommand{\h}{\mathrm{h}}

\newcommand{\rfp}{\mathrm{rfp}}
\newcommand{\Nis}{\mathrm{Nis}}
\newcommand{\cdp}{\mathrm{cdp}}

\newcommand{\Lpf}{\mathrm{L}^{\pf}}
\newcommand{\st}{\mathrm{st}}

\newcommand{\Cat}{\mathscr{C}\mathrm{at}}

\newcommand{\CCC}{\mathscr{C}}

\newcommand{\SSS}{\mathscr{S}}

\newcommand{\PPP}{\mathscr{P}}
\newcommand{\VVV}{\mathscr{V}}
\newcommand{\FFF}{\mathscr{F}}

\DeclareMathOperator{\End}{End}

\newcommand{\Sp}{{\mathscr{S}\mathrm{p}}}

\renewcommand{\Pr}{\mathscr{P}\mathrm{r}}
\newcommand{\PrL}{{\Pr}^\mathrm{L}}
\newcommand{\PrLo}{{\Pr}^{\mathrm{L},\otimes}}
\newcommand{\PrR}{{\Pr}^\mathrm{R}}
\newcommand{\Perf}{{\mathscr{P}\mathrm{erf}}}

\newcommand{\SchFp}{\Sch_{\F_p}}

\newcommand{\PerfFp}{\Perf_{\F_p}}

\newcommand{\Lidh}{\mathrm{L}_{\textup{idh}}}

\newcommand{\Lmot}{\mathrm{L}_{\textup{mot}}}
\newcommand{\Lcdh}{\mathrm{L}_{\textup{cdh}}}
\newcommand{\Lpcdh}{\mathrm{L}_{\pcdh}}
\newcommand{\Lh}{\mathrm{L}_{\h}}
\renewcommand{\L}{\mathrm{L}}
\DeclareMathOperator{\Fun}{Fun}

\newcommand{\PrLotimesst}{\Pr^{\mathrm{L},\otimes}_{\mathrm{st}}}

\newcommand{\Perfpfp}{\Perf^{\pfp}}
\newcommand{\Schfp}{\Sch^{\fp}}
\newcommand{\UH}{\mathrm{UH}}
\newcommand{\SchfpSn}{\Schfp_{S_0}}
\newcommand{\SchfpS}{\Schfp_{S}}

\newcommand{\SchSn}{\Sch_{S_0}}
\newcommand{\PerfS}{\Perf_S}

\newcommand{\PerfpfpS}{\Perfpfp_S}

\newcommand{\rSch}{\mathrm{r}\Sch}

\newcommand{\rSchSo}{\mathrm{r}\Sch_{S_1}}
\newcommand{\rSchFp}{\mathrm{r}\Sch_{\mathbb{F}_p}}
\newcommand{\SchSo}{\Sch_{S_1}}
\newcommand{\rSchrfpSo}{\rSch_{S_1}^{\rfp}}

\DeclareMathOperator{\map}{map}

\newcommand{\N}{\mathbb{N}}
\newcommand{\Z}{\mathbb{Z}}

\newcommand{\A}{\mathbb{A}}
\newcommand{\F}{\mathbb{F}}
\newcommand{\Fp}{\mathbb{F}_p}
\renewcommand{\P}{\mathbb{P}}

\newcommand{\Sm}{\mathscr{S}\mathrm{m}}
\renewcommand{\H}{\mathscr{H}}

\newcommand{\SH}{\mathrm{S}\mathscr{H}}
\newcommand{\SHu}{\underline{\SH}}

\newcommand{\Sh}{\mathscr{S}\mathrm{h}}
\newcommand{\G}{\mathbb{G}}
\newcommand{\T}{\mathbb{T}}
\newcommand{\Psh}{\mathscr{P}}
\newcommand{\oO}{\mathcal{O}}

\DeclareMathOperator{\Spec}{Spec}

\newcommand{\Sch}{\mathscr{S}\mathrm{ch}}

\newcommand{\barA}[1][1]{\overline{\A}{}^{#1}}
\newcommand{\barT}{\overline{\T}}
\newcommand{\barP}[1][1]{\overline{\P}{}^{#1}}

\newcommand{\uH}{\underline{\H}}
\newcommand{\Hpf}{\H^{\pf}}
\newcommand{\uSH}{\underline{\SH}}
\newcommand{\SHpf}{\SH^{\pf}}
\newcommand{\uHpf}{\underline{\H}^{\pf}}
\newcommand{\uSHpf}{\underline{\SH}^{\pf}}
\newcommand{\uphi}{\underline{\varphi}}
\newcommand{\upsi}{\underline{\psi}}
\newcommand{\uPhi}{\underline{\Phi}}
\newcommand{\uPsi}{\underline{\Psi}}

\DeclareMathOperator{\Hom}{Hom}
\newcommand{\Homu}{\underline{\Hom}}

\newcommand{\Frob}{\mathrm{Frob}}

\newcommand{\GL}{\mathrm{GL}}
\newcommand{\B}{\mathrm{B}}

\DeclareMathOperator{\K}{K}

\DeclareMathOperator{\KH}{KH}
\DeclareMathOperator{\KGL}{KGL}
\DeclareMathOperator{\BGL}{BGL}

\DeclareMathOperator{\Mdf}{Mdf}
\DeclareMathOperator{\ZR}{ZR}

\let\hat\widehat

\renewcommand{\to}[1][]{\overset{#1}{\rightarrow}}		
	
\newcommand{\inj}[1][]{\overset{#1}{\hookrightarrow}}		

\theoremstyle{definition}
\newtheorem{Def}{Definition}[section]
\newtheorem{Not}[Def]{Notation}
\newtheorem{Rem}[Def]{Remark}
\newtheorem{Rem*}[]{Remark}
\newtheorem{Exm}[Def]{Example}

\newtheorem*{Conj*}{Conjecture}

\theoremstyle{plain}
\newtheorem{Prop}[Def]{Proposition}
\newtheorem{Thm}[Def]{Theorem}

\newtheorem{Lem}[Def]{Lemma}
\newtheorem{Cor}[Def]{Corollary}

    \newcounter{zaehler}
    
	\newtheorem{introthm}[zaehler]{Theorem}

\newcommand{\citestacks}[1]{\cite[\href{https://stacks.math.columbia.edu/tag/#1}{Tag~#1}]{stacks-project}}

\title{Motivic homotopy theory for perfect schemes}

\author{Christian Dahlhausen}
\address{Institut für Mathematik, Universität Heidelberg, Im Neuenheimer Feld 205, 69121 Heidelberg, Germany}
\email{cdahlhausen@mathi.uni-heidelberg.de}

\author{Jeroen Hekking}
\address{Fakultät für Mathematik, Universität Regensburg, 93040 Regensburg, Germany}
\email{jeroen.hekking@ur.de}

\author{Storm Wolters}
\address{Mathematisch Instituut, Universiteit Leiden, Einsteinweg 55, 2333 CC Leiden, The Netherlands}
\email{s.wolters@math.leidenuniv.nl}

\begin{document}
\begin{abstract}
We construct a perfect version of Morel--Voevodsky's motivic homotopy category over a perfect base scheme in positive characteristic.
By checking the axioms of a coefficient system, we establish a six-functor formalism.
We show that multiplication by $p$ is already invertible in the perfect motivic homotopy catgory.
By work of Elmanto--Khan the functor sending an $\F_p$-scheme $S$ to the category $\SH(S)[1/p]$ is invariant under universal homeomorphisms, hence under perfections. Our result gives an explicit model for the localization of $\SH$ at the universal homeomorphisms, which we conclude is the same as $\SH[1/p]$.
\end{abstract}

\maketitle
\setcounter{tocdepth}{1}
\tableofcontents
\setcounter{tocdepth}{3}

\section*{Introduction}

To any $\F_p$-scheme $S_0$ with perfection $S$, we associate  the \textbf{unstable perfect homotopy category} (Definition~\ref{Def:perfect-H})
\[ \H^\pf(S) \coloneqq \Sh_\mathrm{Nis}^{\barA}(\Sm^\pf_S), \]
which is the full subcategory of $\PPP(\Sm^\pf_S)$ spanned by $\barA$-invariant perfect Nisnevich sheaves; here $\Sm^\pf_S$ denotes the category of perfectly smooth $S$-schemes and $\barA$ is the perfection of $\A^1$.
Subsequently, we define the \textbf{stable perfect motivic homotopy category} (Definition~\ref{Def:perfect-SH})
\[ \SH^\pf(S) \coloneqq \H^\pf(S)_\bullet[(\overline{\P}^1,\infty)\inv] \]
by $\otimes$-inverting the perfection $\barP[1]$ of $\P^1$ (pointed at $\infty$). We establish the following general result.

\begin{introthm}[Corollary~\ref{Cor:SH_motivic_coefficient_systems}] 
The functor which sends an $\Fp$-scheme $S_0$ with perfection $S$ to the $\infty$-category $\SH^\pf(S)$ satisfies smooth base change, the smooth projection formula, the gluing theorem, homotopy invariance, and Thom stability. 
\end{introthm}
In particular, by using Ayoub's thesis, the assignment $S_0\mapsto\SH^\pf(S)$ defines a six-functor formalism on the category $\Sch_{\Fp}$.\footnote{More precisely, in the context of our paper the functor $S_0 \mapsto \SH^\pf(S)$ is a \emph{coefficient system} in the sense of Definition~\cite[App.~A]{DHW1}. This is essentially the same notion as Ayoub's \emph{2-foncteur homotopique stable} \cite[Def.~1.4.1]{AyoubThesis} and yields a six-functor formalism, cf.\ \cite{AyoubThesis}, \cite{CisinskiDeglise}, \cite{liu-zheng}. For reasons of convenience, we often cite Khan's preprint \cite{KhanMorelVoevodsky} which is written in the language of $\infty$-categories. Consider \cite[Rem.~2.18, Rem.~2.21]{GallauerSix} for an overview and history of some of the different approaches to six-functor formalisms available in the literature.}

We relate the classical and the perfect theory via an adjunction (\S\ref{subsec:adjunctions})
\[ (\Phi \dashv \Psi) \colon \SH(S_0) \rightleftarrows \SH^\pf(S),  \]
where $\Phi$ turns out to be a localization (Proposition~\ref{Prop:Psi_ff}). This uses the the pcdh-topology (Definition~\ref{Def:pcdh-topology}), which is an analogue of the cdh-topology for perfect schemes, and the idh-topology (Definition~\ref{Def:idh-topology}---standing for \emph{inseparably decomposed h-topology}), which is a lift of the pcdh topology to schemes that are not necessarily perfect. This yields the category $\SH_\idh(S_0)$ which is an idh-local version of $\SH(S_0)$ (Definition~\ref{Def:SHidh}). We establish the following comparison result.

\begin{introthm}[Theorem~\ref{Thm:comparison-results-all}]
\label{Thm:IntroD}
The functor $\Phi$ induces equivalences 
\[ \SH(S_0)[1/p] \xrightarrow{\sim} \SH_\idh(S_0) \xrightarrow{\sim} \SH^\pf(S).\]
\end{introthm}

Recall that a theorem of Elmanto--Khan says that a universal homeomorphism $f\colon T_0\to S_0$ over $\F_p$ induces an equivalence $f^* \colon \SH(S_0)[1/p] \xrightarrow{\sim} \SH(T_0)[1/p]$ \cite[Cor.~2.1.5]{Elmanto-Khan-perfection}. Our results that $\SH_\idh(S_0) \simeq \SHpf(S)$ and $\SHpf(S) \simeq \SHpf(S)[1/p]$ say that also the converse is true: localizing $\SH$ at the universal homeomorphisms is the same as inverting multiplication by $p$. 

We also mention the study of Frobenius fixed points in motivic homotopy theory by Richarz--Scholbach in \cite{richarz-scholbach-frobenius}. They construct the category $\SH(S/\Frob) \coloneqq \SH(S)^{\Frob^*}$ consisting of pairs $(X,\lambda)$ where $X \in \SH(S)$ and $\lambda$ is an equivalence $X \simeq \Frob^*(X)$. This is not a subcategory of $\SH(S)$ though, since even for $\CCC^{\id^*}$ this fails \cite[Ex.~3.3]{richarz-scholbach-frobenius}. Nor is $p$ in general invertible in $\SH(S/\Frob)$ \cite[Lem.~5.10, Cor.~5.11]{richarz-scholbach-frobenius}. We do obtain a canonical twisting functor $\SH^\pf(S) \to \SH(S/\Frob)[1/p]$ by \cite[Cor.~4.8]{richarz-scholbach-frobenius}, using Theorem~\ref{Thm:IntroD}.

\subsection*{Future work}
The idh-topology makes perfect sense over any base, and coincides with the cdh-topology in characteristic zero.\footnote{In fact, we conjecture that most of the results from Sections~\ref{Sec:Perf}--\ref{Sec:idh} go through over any base after replacing perfect schemes with absolutely weakly normal schemes. The latter are obtained by inverting universal homeomorphisms of reduced schemes in the general case, thus coincide with perfect schemes in characteristic $p$ \cite[Prop.~14.3.2]{BarwickExodromy} \cite[Thm.~1.9]{CarlsonReconstructionschemesetaletopoi}.}
In the equivalence $\SHpf(S)\simeq \SH_\idh(S_0)$ for an $\F_p$-scheme $S_0$ with perfection $S$, the right-hand side thus makes sense for any scheme $S_0$ \emph{over $\Z$}. This comes with a left-adjoint functor $\Phi \colon \SH \to \SH_\idh$ which is invertible in characteristic zero and recovers $\Phi \colon \SH \to \SHpf$ in characteristic $p$.
We conjecture that $\SH_\idh$ induces a six functor formalism over any base.

Once having established this setup, we expect to generalize our duality theorems for modules over K-theory in characteristic zero \cite{DHW1}.
From this point of view, Theorem~A from loc.\ cit.\ and the analogous statement for $\KGL[1/p]$-modules from \cite[Thm.~2.4.9]{BondarkoDeglise2017} should both be incarnations of a general statement on duality for $\KGL$-modules in $\SH_\idh$. Placing these results on equal footing should allow us to combine them into a duality statement in certain mixed characteristic cases, e.g., over a discrete valuation ring.

\subsection*{Leitfaden}
In Section~\ref{Sec:Perf} we recollect some basic facts about perfect schemes and describe the category of perfect schemes as the localization of schemes by inverting universal homeomorphisms, also in the (perfectly) finitely presented case (Proposition~\ref{Prop:perfection_localization}).

In Section~\ref{Sec:pcdh} we introduce and study the $\pcdh$-topology (Definition~\ref{Def:pcdh-topology})---a perfect version of the $\cdh$-topology---and show that the associated sheaf topos is hypercomplete (Corollary~\ref{Cor:pcdh-topos-hypercomplete}). In contrast to the $\cdh$-topology, the $\pcdh$-topology is subcanonical (Corollary~\ref{Cor:pcdh-subcanonical}).

In Section~\ref{Sec:idh} we introduce and study the $\idh$-topology (Definition~\ref{Def:idh-topology}). This is a lift of the $\pcdh$-topology to schemes over $\Fp$ such that  $\idh$-sheaves are precisely those $\cdh$-sheaves that invert universal homeomorphisms (Proposition~\ref{Prop:cdh_to_idh}). The $\idh$-sheaf topos is equivalent to the $\pcdh$-sheaf topos (Corollary~\ref{Cor:idh=pcdh}) and the $\idh$-points are precisely the spectra of perfect Henselian valuation rings (Proposition~\ref{Prop:points-of-idh-topology}).

In Section~\ref{Sec:perfect-motivic-cat} we introduce the unstable (Definition~\ref{Def:perfect-H}) and stable (Definition~\ref{Def:perfect-SH}) perfect motivic homtopy theory. We continue the study in Section~\ref{Sec:perfect-motivic}
and show that the stable category is connected by an adjunction to the classical category (Proposition~\ref{Prop:Phi_comm_susp}). The associated right adjoint turns out to be fully faithful (Proposition~\ref{Prop:Psi_ff}). 

Our main results are the establishment of perfect motivic homotopy theory as a coefficient system (Theorem~A = Corollary~\ref{Cor:SH_motivic_coefficient_systems}) and the comparison results from Theorem~B in section~\ref{Sec:comparison}.
Finally, we study representability of connective algebraic K-theory in the unstable perfect motivic homotopy category (Proposition~\ref{Prop:K-connective-lies-in-Hpf}) and describe its representing object (Corollary~\ref{Cor:K=ZxBGL}).

\subsection*{Acknowledgements}
The authors cordially thank Denis-Charles Cisinski for giving helpful input, and for valuable discussions and constructive feedback.
We thank Tom Bachmann for pointing out to us that inverting the Frobenius already should invert $p$.
Moreover, the authors thank
Tess Bouis,
Fangzhou Jin,
Robin de Jong, 
Adeel Khan,
Niklas Kipp, 
Sebastian Wolf, and
Can Yaylali 
for helpful discussions related to this paper's content.

JH has been generously supported by the Knut and Alice Wallenberg Foundation, project number 2021.0287.
The project was jointly supported by the Deutsche Forschungsgemeinschaft (DFG).
Namely, JH is supported by the Collaborative Research Centre SFB 1085 \emph{Higher Invariants -- Interactions between Arithmetic Geometry and Global Analysis}, project number 224262486, and CD is supported by the Collaborative Research Centre TRR 326 \textit{Geometry and Arithmetic of Uniformized Structures (GAUS)}, project number 444845124. Moreover, JH and SW thank GAUS for financing a work stay in Heidelberg in August 2024.
CD thanks Elden Elmanto for hosting him at the University of Toronto in September 2024.

 \section*{Conventions}
\subsection*{Mapping objects}
Let $\VVV$ be a closed symmetric monoidal $\infty$-category. For a $\VVV$-enriched $\infty$-category $\CCC$, we write $\Hom_\CCC(X,Y)_\VVV \in \VVV$ for the mapping object in $\VVV$ of $\CCC$-morphisms $X \to Y$, where $X,Y \in \CCC$. For short, we write $\Hom_\CCC(X,Y) \coloneqq \Hom_\CCC(X,Y)_\SSS$ for the ordinary mapping space. If $\CCC$ is enriched over itself, we write $\Homu_\CCC(X,Y) \coloneqq \Hom_\CCC(X,Y)_{\CCC}$ for the inner hom. In the case $\VVV = \Sp$ is the $\infty$-category of spectra, we also write $\map(X,Y) \coloneqq \Hom_\CCC(X,Y)_{\Sp}$.

\subsection*{Assumptions}
For the entire article, let $p$ be a prime number.
Unless otherwise stated, we assume that all schemes (hence all morphisms) are quasi-compact and quasi-separated (qcqs).

\subsection*{Notation}
Let $S$ be a scheme, and $\Sch_S$ the category of schemes over $S$. 
\begin{enumerate}
    \item Write $\PrLo$ for the $\infty$-category of  presentably symmetric monoidal $\infty$-categories, with colimit-preserving, symmetric monoidal functors as morphisms. Write $\PrLotimesst \subset \PrLo$ for the full subcategory spanned by $\CCC \in \PrLo$ for which the underlying $\infty$-category is stable.
   \item Write $\Schfp_S \supset \Sm_S$ for the full subcategories of $\Sch_S$ spanned by finitely presented resp.\ smooth and finitely presented $S$-schemes.\footnote{Be aware that smooth morphisms are, in general, not finitely presented, only locally finitely presented. But redefining $\Sm_S$ with this extra assumption does no harm, since all our schemes are qcqs, cf.\ \cite[Prop.~C.5]{hoyois-quadratic}.}
    \item Write $\H_{\bullet}(S)$ for the motivic homotopy category of $\A^1$-invariant, pointed Nisnevich sheaves on $\Sm_S$ and write $\SH(S)$ for the stable motivic homotopy category, obtained by tensor-inverting the Tate sphere $\T \coloneqq (S^1,1)\otimes(\G_m,1) \in \H_\bullet(S)$. Recall that $\T \simeq (\P^1,\infty)$ in $\H_\bullet(S)$.
    \item Let $X \mapsto X_+$ be the the obvious functor $\Sm_S \to \H_{\bullet}(S)$ induced by adding a basepoint and localization.
    \item Put $\Sigma_{\T} \coloneqq (-) \wedge \T$ and $\Omega_{\T} \coloneqq \Homu_{\H_\bullet(S)}(\T,-)$. We get an adjunction
    \[ (\Sigma_{\T} \dashv \Omega_{\T}) \colon \H_\bullet(S) \rightleftarrows \H_\bullet(S). \]  
    \item $\Sigma_{\T} \dashv \Omega_{\T}$ induces an adjuntion
    \[ (\Sigma^\infty \dashv \Omega^\infty) \colon \H_\bullet(S) \rightleftarrows \SH(S). \]
\end{enumerate}

We write $h \colon \CCC \to \PPP(\CCC)$ for the Yoneda embedding into the $\infty$-category of presheaves on $\CCC$ (always space-valued unless otherwise stated).

Recall that a family $\{X_i \to Y\}_{i \in I}$ is \emph{completely decomposed} if for each $y \in Y$ there is some $i \in I$ and $x \in X_i$ over $y$ such that the induced field extension $\kappa(y) \subset \kappa(x)$ is an isomorphism. 

A \emph{Nisnevich square} is a Cartesian square
\begin{center}
        \begin{tikzcd}
            W \arrow[r] \arrow[d] & V \arrow[d, "p"] \\
            U \arrow[r, "j"] & X
        \end{tikzcd}
\end{center}
where $j$ is an open immersion and $p$ is \'{e}tale such that the canonical map $(X \setminus U)_\red \times_X V \to (X \setminus U)_\red$ is an isomorphism. The \emph{Nisnevich topology} on $\Sch$ is generated by the cd-structure of Nisnevich squares. Since all schemes have been assumed to be qcqs, the Nisnevich topology is generated by completely decomposed \'{e}tale covering families \cite[Appendix~A]{BachmannNorms}.
The \emph{cdp-topology} is generated by completely decomposed families of  proper maps. Also, the \emph{cdh-topology} is generated by the cdp-topology and the Nisnevich topology.\footnote{Without mention, we pass to the restriction of these topologies to (sub)categories of schemes over a base in the obvious way. Note we do not require covering maps for the cdp-topology to be finitely presented, hence this differs from \cite{ehik--milnor-excision}. Clearly, after restricting to $\SchfpS$ this difference disappears.}

\section{Perfect schemes}
\label{Sec:Perf}

Recall that a scheme $X$ over $\F_p$ is \emph{perfect} if the Frobenius endomorphism is an equivalence. Note that the inclusion of perfect schemes into the category of all schemes over $\F_p$ has a right adjoint. The induced endofunctor on the category of schemes over $\F_p$ is called the \emph{perfection}, written $X \mapsto X^{\pf}$. Concretely, one can calculate $X^{\pf}$ as the limit of the iterated Frobenius maps $\ldots\to X \to X \to X$.  We will use a few results from \cite{BhattProjectivity} throughout, which we record here for convenience. 

\begin{Lem}[{\cite[Lem.\ 3.4]{BhattProjectivity}}]
	\label{Lem:Bhatt.3.4}
Let $f:X \to Y$ be a morphism of (not necessarily qcqs) schemes over $\F_p$. Then the following properties hold for $f$ if and only if they hold for $f^\pf$: quasi-compact; quasi-separated; separated; universally closed. Moreover, if one of the following properties holds for $f$, then it also holds for $f^\pf$: \'{e}tale; a closed immersion; an open immersion.
\end{Lem}

\begin{Not}
    Throughout this paper and unless stated otherwise, $S_0$ is a scheme over $\Fp$ with underlying reduced scheme $S_1$ and perfection $S$.
\end{Not}

Write $\Perf_S$ for the category of perfect schemes over $S$.
Let a morphism $f \colon X \to S$ between perfect schemes be given. Following \cite[Prop.~3.11]{BhattProjectivity}, we say that $f$ is \emph{perfectly of finite presentation} if there is a finitely presented model $f_0 \colon X_0 \to S$. Write $\Perfpfp_S$ for the full subcategory of $\Perf_S$ spanned by the perfectly finitely presented schemes over $S$.

Taking limits of cofiltered inverse sytems of schemes with affine transition maps commutes with perfection \cite[Prop.~3.12]{BhattProjectivity}.\footnote{Recall, by assumption all schemes are assumed to be qcqs unless stated otherwise.} 

\begin{Lem}[{\cite[Lem.\ 3.8]{BhattProjectivity}}]
	\label{Lem:Bhatt.3.8}
	A morphism $f: X \to Y$ in $\Sch_{\F_p}$ is a universal homeomorphism if and only if $f^\pf$ is an isomorphism.
 \end{Lem}

\subsection{Passing to reduced schemes} 
Write $\rSchFp$ for the full subcategory of $\SchFp$ spanned by reduced schemes.\footnote{The reason for passing to reduced schemes is that then universal homeomorphisms allow a calculus of fractions also in the `reducedly finitely presented' case, see Lemma~\ref{Lem:uh_is_limit_fpuh}. This is an essential ingredient in showing that the perfection functor is a localization also in the finitely presented case, see Proposition~\ref{Prop:perfection_localization}.} Recall that the inclusion $\rSchFp \subset \SchFp$ admits a right adjoint which sends $X \in \SchFp$ to the underlying reduced scheme $X_\red$. Since the counit $X_\red \to X$ is a nilimmersion, the perfection functor $(-)^\pf \colon \SchFp \to \PerfFp$ factors through the reduction functor $(-)_\red \colon \SchFp \to \rSchFp$. The obvious relative version is a factorization
\[ \SchSn \xrightarrow{(-)_\red} \rSchSo \xrightarrow{(-)^{\pf}} \PerfS \]
of the perfection $(-)^{\pf} \colon \SchSn \to \PerfS$.

\begin{Lem}
\label{Lem:rSch_cofiltered_lims}
    The full subcategory $\rSchSo \subset \SchSo$ is closed under limits of cofiltered inverse systems of reduced schemes with affine transition maps.
\end{Lem}

\begin{proof}
    Let $W = \lim W_\alpha$ for $\{W_\alpha\}_{\alpha \in A}$ a cofiltered inverse system of reduced schemes with affine transition maps, where the limit is taken in $\SchSo$. We claim that $W$ is reduced. Since this is a local property, we may assume without loss of generality that $W$, hence all $W_\alpha$, are affine \cite[\href{https://stacks.math.columbia.edu/tag/01Z6}{Tag 01Z6}]{stacks-project}. For any $f \in \oO_W$ there is some $f_\alpha \in \oO_{W_\alpha}, \alpha \in A$ such that $f$ is the image of $f_\alpha$ under the transition map $\oO_{W_\alpha} \to \oO_W$. Moreover, $f$ is nilpotent if and only if $f_\alpha$ is nilpotent for big enough $\alpha$. The claim follows.
\end{proof}

\begin{Def}
    A morphism $f \colon X \to Y$ in $\rSchFp$ is \emph{reducedly finitely presented} if for any cofiltered system $\{W_\alpha\}_{\alpha \in A}$ in $\rSch_Y$ with affine transition maps and limit $W$, the natural map
    \[ \colim \Hom_{\rSch_Y}(W_\alpha ,X) \to \Hom_{\rSch_Y}(W,X) \]
    is an equivalence. Write $\rSchrfpSo$ for the category of reducedly finitely presented schemes over $S_1$.
\end{Def}

\begin{Rem}  
\label{Rem:fp_rfp_pfp}
    By adjunction and definition, it is clear that
    $ (-)_\red \colon \SchFp \to \rSchFp$
    sends finitely presented morphisms to reducedly finitely presented morphisms. Likewise, the functor
    $(-)^{\pf} \colon \rSchFp \to \PerfFp$
    sends reducedly finitely presented morphism to perfectly finitely presented morphisms.
\end{Rem}

\begin{Exm}
    Not every reducedly finitely presented morphism is finitely presented. Indeed, define
    \[ A \coloneqq \frac{\Fp[x_i \mid i \in \N]}{(x_{i+1}^2 - x_i)_{i \geq 0}},\]
    and put $B \coloneqq A/(x_0)$. Then $A$ is reduced, the immersion $\Spec 
    B\to \Spec A$ is finitely presented, and thus $\Spec B_\red\to \Spec A$ is reducedly finitely presented. But $B_\red \cong \F_p$, so $\Spec B_\red\to \Spec A$ is not finitely presented.
\end{Exm}

\begin{Rem}
\label{Rem:rfp_local}
It follows straightforwardly from the definition that being reducedly finitely presented is stable under base change and composition, and that it satisfies cancellation in the sense that morphisms in $\rSchrfpSo$ are reducedly finitely presented. As in the ordinary case, one shows that it is also Zariski-local on the target. 
\end{Rem}

\begin{Prop}
\label{Prop:rfp_factorization}
    Let $f \colon X \to Y$ in $\rSchFp$ be given. Then $f$ is reducedly finitely presented if and only if Zariski-locally on $Y$ there is a factorization 
    \[ X \xrightarrow{g} X_0 \xrightarrow{f'} Y  \]
     of $f$ such that $g$ is a nilimmersion and $f'$ is finitely presented.
\end{Prop}

\begin{proof}
    Using Remark \ref{Rem:rfp_local}, the same strategy as the proof of \cite[Prop.~3.11]{BhattProjectivity} goes through.
\end{proof}

Since closed immersions are of finite type, we obtain:

\begin{Cor}
\label{Cor:rfp_ft}
    Reducedly finitely presented maps are of finite type.
\end{Cor}

Since the nilradical of a Noetherian scheme is of finite type, Proposition \ref{Prop:rfp_factorization} yields the following:

\begin{Cor}
\label{Cor:rpf_is_lfp_Noetherian}
    If $S_1$ is Noetherian then $X_1 \to S_1$ with $X_1$ reduced is finitely presented if and only if it is reducedly finitely presented.
    
    Hence, in this case, the reduction functor
    \[ (-)_\red \colon \SchfpSn \to \rSchrfpSo \simeq \rSch^{\fp}_{S_1} \]
    is a coreflective localization.
\end{Cor}

The reduction functor $(-)_\red \colon \SchfpSn \to \rSchrfpSo$ preserves Nisnevich and cdp-covering families, hence cdh-covering families. For $\tau \in \{\Nis,\cdp,\cdh\}$ we obtain an adjunction
\[ (\alpha_\tau \dashv \beta_\tau) \colon \Sh_\tau(\SchfpSn) \rightleftarrows \Sh_\tau(\rSchrfpSo), \]
where $\beta_\tau$ is precomposition with $(-)_\red$.

\begin{Prop}
\label{Prop:Noeth_cdp_cdh}
If $S_0$ is Noetherian, then $\alpha_\tau \dashv \beta_\tau$ is an adjoint equivalence for $\tau \in\{\cdp,\cdh\}$.
\end{Prop}

\begin{proof}
    By Corollary \ref{Cor:rpf_is_lfp_Noetherian} the reduction functor $(-)_\red$ is a coreflective localization. Hence, for $F \in \Sh_\tau(\SchfpSn)$ it holds that $F((-)_\red) \in \Sh_\tau(\rSchrfpSo)$. For any $X \in \SchfpSn$ it holds that $X_\red \to X$ is a monic $\tau$-cover in $\SchfpSn$, hence that $F(X) \to F(X_\red)$ is an equivalence. It follows that $\beta_\tau(F((-)_\red)) \simeq F$, hence that $\beta_\tau$ is essentially surjective.

    Again by Corollary \ref{Cor:rpf_is_lfp_Noetherian}, $\beta_\tau$ is given by precomposition with a localization, hence is fully faithful.
\end{proof}

\subsection{Perfection as localization}
Lemma \ref{Lem:Bhatt.3.8} suggests that taking perfections should be the same as localizing at universal homeomorphisms. In this subsection we realize this idea relative to the base $S \to S_1$, also after restricting to reducedly (resp.\ perfectly) finitely presented schemes.
\begin{Prop}
\label{Prop:perfection_right_Bousfield}
    The functors
    \[ \SchSn \xrightarrow{(-)_\red} \rSchSo \xrightarrow{(-)^{\pf}} \PerfS \]
    are coreflective localizations.
\end{Prop}

\begin{proof}
    This is immediate from the adjunctions. The left adjoints are given by postcomposition with $S \to S_1 \to S_0$.
\end{proof}

For $\CCC \subset \Sch_B$ a category of schemes over a given base $B$, write $\CCC[\UH^{-1}]$ for the  localization of $\CCC$ obtained by formally inverting the universal homeomorphisms in $\CCC$. 

\begin{Cor}
The perfection functor $(-)^{\pf} \colon \Sch_{S_0} \to \Perf_S$ is canonically equivalent to the localization $\Sch_{S_0} \to \Sch_{S_0}[\UH\inv]$. Likewise for $(-)^{\pf} \colon \rSch_{S_1} \to \Perf_S$
\end{Cor}

\begin{proof}
This is immediate from Lemma \ref{Lem:Bhatt.3.8} and Proposition \ref{Prop:perfection_right_Bousfield}.
\end{proof}

\begin{Lem}
\label{Lem:uh_is_limit_fpuh}
    Let $u \colon W \to X$ be a universal homeomorphism between reduced schemes. Then $W$ is a cofiltered limit of reduced $X$-schemes $W_\alpha$ for which $u_\alpha \colon W_\alpha \to X$ is a universal homeomorphism, reducedly of finite presentation.
\end{Lem}

\begin{proof}
    By \cite[Lem.~2.2]{BarwickTopological}, we can write $W$ as a cofiltered limit of $X$-schemes $W_\alpha'$ for which $u_\alpha' \colon W_\alpha' \to X$ is a finite universal homeomorphism. By \cite[Prop.~2.3]{BarwickTopological}, each $u_\alpha'$ factors as 
    \[ W_\alpha' \xrightarrow{v_\alpha} W_\alpha'' \xrightarrow{u_\alpha''} X \]
    such that $v_\alpha$ is a nilimmersion and $u_\alpha''$ is of finite presentation. Put $W_\alpha \coloneqq (W_\alpha')_\red \simeq (W_\alpha'')_\red$. Since $W$ is reduced, it holds $W \simeq \lim W_\alpha$ by Lemma \ref{Lem:rSch_cofiltered_lims}. And $u_\alpha \colon W_\alpha \to X$ is reducedly finitely presented by Remark \ref{Rem:fp_rfp_pfp}.
\end{proof}

The class of universal homeomorphisms in $\Sch_{S_0}$ admits a calculus of fractions.\footnote{In other words, universal homeomorphisms form a \emph{right multiplicative system}, cf.\ \citestacks{04VB}.} Indeed, the right Ore condition is immediate by base change. To see the right cancellable condition, consider a diagram
\begin{center}
    \begin{tikzcd}
        X \arrow[r, shift left, "f"] \arrow[r, shift right, "g", swap] & Y \arrow[r, "v"] & Z
    \end{tikzcd}
\end{center}
in $\Sch_{S_0}$ such that $v$ is a universal homeomorphism and such that $vf=vg$. Then $f^{\pf} = g^{\pf}$, hence precomposing with the perfection morphism $X^{\pf} \to X$ coequalizes $f,g$. The following shows that this also works in the reducedly finitely presented case.

\begin{Lem}
\label{Lem:UH_calculus_fractions}
    The class of universal homeomorphisms in $\rSchrfpSo$ admits a calculus of fractions.
\end{Lem}

\begin{proof}
    Again it suffices to show the right cancellable condition. So let
    \begin{center}
    \begin{tikzcd}
        X \arrow[r, shift left, "f"] \arrow[r, shift right, "g", swap] & Y \arrow[r, "v"] & Z
    \end{tikzcd}
\end{center}
    in $\rSchrfpSo$ be given such that $vf=vg$. Let $u \colon X^{\pf} \to X$. As before, it holds $fu=gu$. Using Lemma~\ref{Lem:uh_is_limit_fpuh}, write $X^{\pf}$ as a cofiltered limit of $X$-schemes $W_\alpha$ for which $u_\alpha \colon W_\alpha \to X$ is a reducedly finitely presented universal homeomorphism.
    Since $Y$ is of reduced finite presentation over $S_1$, the identity $fu = gu$ is witnessed by an identity $fu_\alpha = gu_\alpha$ for some $\alpha$. Taking such $u_\alpha \colon W_\alpha \to X$ finishes the argument.
\end{proof}

\begin{Lem}
\label{Lem:descent_pfp_model}
   Let $f \colon X \to S$ be perfectly of finite presentation. Then there is a reducedly finitely presented morphism $f_1 \colon X_1 \to S_1$ such that $f_1^{\pf} = f$.
\end{Lem}

\begin{proof}
    Take a finitely presented model $X_0 \to S$ of $f \colon X\to S$ with associated universal homeomorphism $X \to X_0$. By Lemma~\ref{Lem:uh_is_limit_fpuh}, write $S$ as a cofiltered limit of  $S_1$-schemes $S_\alpha$ for which the structure map $S_\alpha \to S_1$ is a reducedly finitely presented universal homeomorphism. By \cite[\href{https://stacks.math.columbia.edu/tag/01ZM}{Tag 01ZM}]{stacks-project}, we find some $\alpha$ and some finitely presented scheme $X_{0 \alpha} \to S_\alpha$ which pulls back to $X_0 \to S$ along the projection $S \to S_\alpha$. Let $X_1 \to X_{1\alpha}$ be the result of applying $(-)_\red$ to the projection $X_0 \to X_{0 \alpha}$. Then the composition $X_{1 \alpha} \to S_\alpha \to S_1$ is reducedly finitely presented, and is sent to $f$ under perfection.
\end{proof}

Consider the commutative diagram
\begin{center}
	\begin{tikzcd}
		\rSchrfpSo \arrow[r, "\gamma'"] \arrow[d, hook]
		& \rSchrfpSo[\UH^{-1}] \arrow[r, "\sigma'"] \arrow[d, "\bar{\iota}"]
		& \Perfpfp_S \arrow[d, hook] \\
		\rSch_{S_1} \arrow[r, "\gamma"]
		& \rSch_{S_1}[\UH\inv] \arrow[r, "\sigma"]
		& \Perf_S
	\end{tikzcd}
\end{center}
where $\gamma',\gamma$ are the localizations, and $\sigma',\bar{\iota},\sigma$ are induced by the universal properties of the localizations. We have seen that $\rSch_{S_1} \to \Perf_S$ is a localization at the universal homeomorphisms, i.e., that $\sigma$ is an equivalence. The following shows that also $\rSchrfpSo \to \PerfpfpS$ is a localization at the universal homeomorphisms:

\begin{Lem} \label{Lem:FiniteType-UH-localisation}
    The functor $\sigma'$ is an equivalence.
\end{Lem}

\begin{proof}
    Since $\sigma'$ is induced by taking perfections, it is essentially surjective by Lemma~\ref{Lem:descent_pfp_model}. It thus suffices to show that the functor $\bar{\iota}$ is fully faithful. To this end, let $X,Y \in \rSchrfpSo$ be given. Using Lemma \ref{Lem:UH_calculus_fractions}, we identify morphisms in our localizations with equivalence classes of spans.\footnote{Recall that two spans are equivalent when they are dominated by a third.} 

    Consider a span
    \[ X \xleftarrow{v} W \xrightarrow{p} Y \]
    in $\Sch_{S_1}[\UH\inv]$, hence with $v$ a universal homeomorphism. By Lemma~\ref{Lem:uh_is_limit_fpuh}, we can write $W$ as a cofilitered limit of reducedly finitely presented $X$-schemes $W_\alpha$ for which $W_\alpha \to X$ is a universal homeomorphism. Since $Y \to S_1$ is reducedly finitely presented, the morphism $p$ factorizes through some $p_\alpha \colon W_\alpha \to Y$. Hence we can replace the given span by the span induced by $W_\alpha$. Now $W_\alpha$ is reducedly finitely presented over $S_1$, which implies fully faithfulness.
\end{proof}

\begin{Rem}
\label{Rem:Noetherian_induction}
    In what follows, we will be using the following familiar set-up for Noetherian induction. Write $S_0$ as cofiltered limit of finitely presented $\Fp$-schemes $S_{0\alpha}$ and put $S_\alpha \coloneqq (S_{0 \alpha})_{\pf}$. Observe that $S \simeq \lim S_\alpha$, and that each $S_{0 \alpha}$ is Noetherian. Recall that there are canonical equivalences of categories
    \begin{align*}
        \SchfpSn \simeq \colim_\alpha \Schfp_{S_{ 0 \alpha }}, &&  \PerfpfpS \simeq \colim_\alpha \Perfpfp_{S_{\alpha}},
    \end{align*}
in $\Cat_\infty$ and hence equivalences of presheaf categories
\begin{align*}
        \PPP(\SchfpSn) \simeq \lim_\alpha \PPP(\Schfp_{S_{0 \alpha }}), &&  \PPP(\PerfpfpS) \simeq \lim_\alpha \PPP(\Perfpfp_{S_{\alpha}}),
\end{align*}
in $\Cat_\infty$.
    In this situation, we have the following descent result by \cite[Thm.~8.10.5]{EGA4}. Let $\gamma$ be a Nisnevich square (resp.\ an abstract blow-up square) in $\SchfpSn$. Then there is some $\alpha$ and some Nisnevisch square (resp.\ abstract blow-up square) $\gamma_\alpha$ in $\Schfp_{S_{0 \alpha}}$ which pulls back to $\gamma$ along $S_0 \to S_{0 \alpha}$. Likewise for perfect Nisnevich squares (resp.\ perfect abstract blow-up squares) in $\PerfpfpS$. 
\end{Rem}

\begin{Prop}
\label{Prop:perfection_localization}
    The perfection functor $\rho \colon \SchfpSn \to \PerfpfpS$ is a localization at the universal homeomorphism $\SchfpSn[\UH\inv] \simeq \PerfpfpS$. 
\end{Prop}

\begin{proof}
   Write $S_0 \simeq \lim S_{0 \alpha}$ as in Remark~\ref{Rem:Noetherian_induction}. By \cite[Thm.~8.10.5]{EGA4} it holds that $f \colon X \to Y$ in $\SchfpSn$ is a universal homeomorphism if and only if there is some $\alpha$  and some universal homeomorphism $f_\alpha \colon X_\alpha \to Y_\alpha$ in $\Schfp_{S_{0 \alpha }}$ which pulls back to $f$ along $S_0 \to S_{0 \alpha}$. For any $\infty$-category $\CCC$ it follows that the equivalence $\Fun(\SchfpSn,\CCC) \simeq \lim \Fun(\Schfp_{S_{0 
   \alpha}},\CCC)$ restricts to an equivalence
   \[ \Fun^{\UH}(\SchfpSn,\CCC) \simeq \lim \Fun^{\UH}(\Schfp_{S_{0 
   \alpha}},\CCC), \]
   where $\Fun^{\UH}(-,-)$ means functors that invert universal homeomorphisms. By Corollary~\ref{Cor:rpf_is_lfp_Noetherian} and Lemma~\ref{Lem:FiniteType-UH-localisation} the claim now follows directly from the universal property of localizations.
\end{proof}

\begin{Rem}
    When $S_0$ is the spectrum of a perfect field, then Proposition~\ref{Prop:perfection_localization} recovers \cite[Lem.~1.32]{CarlsonReconstructionschemesetaletopoi} (c.f.\ \cite[Lem.~2.9]{richarz-scholbach}).
    
    When for any $X \in \SchfpSn$ it holds that the absolute Frobenius on $X$ is of finite presentation, then the argument in \cite{CarlsonReconstructionschemesetaletopoi} also goes through to show that the canonical map
    \[ \SchfpSn[\Frob^{-1}] \to \SchfpSn[\UH^{-1}] \]
    is invertible, where the decoration $[\Frob^{-1}]$ means we are inverting all absolute Frobenii in $\SchfpSn$.
    This assumption holds if and only if the absolute Frobenius on $S_0$ itself is of finite presentation, for example when $S_0$ is of finite type over a perfect Noetherian base.  \cite[\href{https://stacks.math.columbia.edu/tag/0CCD}{Tag 0CCD}]{stacks-project}.
\end{Rem}

\section{The pcdh-topology}
\label{Sec:pcdh}
When emphasizing the perfect setting, we call the restriction of the Nisnevich topology to $\PerfFp$ the \emph{perfect} Nisnevich topology. Similarly for the terms \emph{perfect} Nisnevich square and \emph{perfect} Nisnevich sheaf. Thus, as in the ordinary case, a Cartesian square 
\begin{equation}
\label{Eq:pfNis_square}
    \begin{tikzcd}
        W \arrow[r] \arrow[d] & V \arrow[d, "p"] \\
        U \arrow[r, "j"] & X
    \end{tikzcd}
\end{equation}
in $\PerfpfpS$ is a \emph{(perfect) Nisnevich square} if $j$ is an open immersion 
and $p$ is \'{e}tale such that $p^{-1}((X\setminus U)_\red) \to (X\setminus U)_\red$ is an isomorphism.

\begin{Lem}
\label{Lem:pfNis_excission}
    A presheaf on $\PerfpfpS$ satisfies Nisnevich excision if and only if it is a sheaf for the Nisnevich topology.\footnote{Recall that a presheaf is \emph{Nisnevich excisive} if it sends Nisnevich squares to Cartesian squares and the empty scheme to the point.}
\end{Lem}

\begin{proof}
    This is \cite[Prop.~3.8]{HoyoisSix}. To be sure, the requirements are satisfied by topological invariance \cite[Thm.~3.7]{BhattProjectivity}.
\end{proof}

\begin{Lem}
\label{Lem:Nis_square_lift}
    Any perfect Nisnevich square in $\PerfpfpS$ is the image under perfection of a Nisnevich square in $\SchfpSn$. 
\end{Lem}

\begin{proof}
    Let a perfect Nisnevich square as in (\ref{Eq:pfNis_square}) be given. Using Corollary~\ref{Lem:descent_pfp_model}, take $X_1 \in \rSchrfpSo$ such that $X_1^{\pf} = X$. By topological invariance, we obtain an open immersion $j_1 \colon U_1 \to X_1$ and an \'{e}tale morphism $p_1 \colon V_1 \to X_1$ such that the perfection of the Cartesian square
    \begin{equation}
    \label{Eq:Nis_lift}
        \begin{tikzcd}
            W_1 \arrow[r] \arrow[d] & V_1 \arrow[d, "p_1"] \\
            U_1 \arrow[r, "j_1"] & X_1
        \end{tikzcd}
    \end{equation}
    is the given square (\ref{Eq:pfNis_square}) \cite[Thm.~3.7]{BhattProjectivity}. Let $Z_1$ be the reduced complement of $j_1$. The question whether $p_1^{-1}(Z_1) \to Z_1$ is an equivalence is a purely topological question, hence can be checked up to perfection, in which case we recover the map $p^{-1}(Z) \to Z$ where $Z$ is the perfect complement of $j$. 

    At this point, we have lifted any perfect Nisnevich square in $\PerfpfpS$ to a Nisnevich square in $\rSchrfpSo$ along the perfection. To obtain a lift to $\SchfpSn$, write $S_0$ as cofiltered limit of finitely presented $\Fp$-schemes $S_{0 \alpha}$ as in Remark \ref{Rem:Noetherian_induction}, and use descent for Nisnevich squares together with Corollary~\ref{Cor:rpf_is_lfp_Noetherian}.
\end{proof}

\begin{Prop}
\label{Prop:pNis_vs_Nis}
    A presheaf $F \in \PPP(\PerfpfpS)$ is a perfect Nisnevich sheaf if and only if $F((-)^{\pf}) \in \PPP(\SchfpSn)$ is a Nisnevich sheaf.
\end{Prop}

\begin{proof}
    If $F$ is a perfect Nisnevich sheaf then $F((-)^{\pf})$ is a Nisnevich sheaf, since perfection preserves the property of being 
    a Nisnevich covering family. By Lemma \ref{Lem:pfNis_excission}, the converse follows from Lemma \ref{Lem:Nis_square_lift}.
\end{proof}

We will now study a topology on perfect schemes analogous to the cdh-topology. A morphism in $\PerfFp$ is called \emph{perfectly proper} if it is the perfection of a proper map. A perfectly finitely presented map $f \colon X \to Y$ is perfectly proper if and only if any finitely presented model $f_0 \colon X_0 \to Y$ is proper, hence this agree with \cite[Def.~3.14, Cor.~3.15]{BhattProjectivity}.

\begin{Def}
    The \emph{pcdp-topology} on $\PerfpfpS$ is generated by completely decomposed families of perfectly proper maps.
\end{Def}

Recall that the cdp-topology on $\SchfpSn$ is associated with the cd-structure consisting of abstract blow-up squares \cite[Prop.~2.1.5]{ehik--milnor-excision}. We show an analogous result for the pcdp-topology.

\begin{Def}
\label{Def:pcdh-square}
    A Cartesian square
    \begin{equation}
    \label{Eq:pablup_square}
        \begin{tikzcd}
            W \arrow[r] \arrow[d] & Y \arrow[d, "p"] \\
            Z \arrow[r, "i"] & X
        \end{tikzcd}
    \end{equation}
    in $\PerfpfpS$ is a \emph{perfect abstract blow-up square} if $i$ is a closed immersion, and $p$ is perfectly proper such that $p^{-1}(X \setminus Z) \to X \setminus Z$ is an ismorphism.
\end{Def}

\begin{Lem}
\label{Lem:cd_splittings_sequence}
    Let $f \colon X \to Y$ perfectly of finite presentation in $\PerfFp$ be given. Then $f$ is completely decomposed if and only if it admits a splitting sequence of perfecty finitely presented, perfect closed subschemes of $Y$.
\end{Lem}

\begin{proof}
    One direction is clear. Conversely, suppose that $f \colon X \to Y$ is completely decomposed. Take a finitely presented model $f_0 \colon X_0 \to Y$ in $\SchFp$. Observe that $f_0$ is completely decomposed, hence admits a splitting sequence by finitely presented closed subschemes $Z_i \subset Y$ \cite[Lem.~2.1.2]{ehik--milnor-excision}. Taking perfections of the $Z_i$ induces a splitting sequence for $f$ by perfectly finitely presented, perfect closed subschemes.
\end{proof}

\begin{Prop}
\label{Prop:pcdp_via_pblups}
    The perfect abstract blow-up squares in $\PerfpfpS$ generate as cd-structure the pcdp-topology on $\PerfpfpS$.  
\end{Prop}

\begin{proof}
    Using Lemma~\ref{Lem:cd_splittings_sequence}, a similar argument as in the proof of \cite[Pop.~2.1.5]{ehik--milnor-excision} goes through.
\end{proof}

\begin{Def}
    \label{Def:pcdh-topology}
    The \emph{pcdh-topology} is the topology on $\PerfpfpS$ generated by the Nisnevich topology and the pcdp topology. 
\end{Def}

\begin{Lem}
\label{Lem:abs_lift}
    Any perfect abstract blow-up square in $\PerfpfpS$ is the image under perfection of an abstract blow-up square in $\SchfpSn$.
\end{Lem}

\begin{proof}
    Let a perfect abstract blow-up square as in (\ref{Eq:pablup_square}) be given. Using Corollary~\ref{Lem:descent_pfp_model} twice, take $X_1 \in \rSchrfpSo$ such that $X_1^{\pf} = X$, and then take $Y_1 \in \rSch^{\rfp}_{X_1}$ such that $Y_1^{\pf} = Y$. Then $Y_1 \to X_1$ is proper: it is of finite type by Corollary \ref{Cor:rfp_ft}, and it is separated and universally closed by Lemma~\ref{Lem:Bhatt.3.4}. Let $i_1 \colon Z_1 \to X_1$ be the reduced closed subscheme on the closed subset $\lvert Z \rvert \subset \lvert X \rvert = \lvert X_1 \rvert$. Clearly we have  $i_1^{\pf} = i$. Hence we obtain a Cartesian square 
    \begin{equation}
    \label{Eq:lift_abstract_blup}
          \begin{tikzcd}
            W_1 \arrow[r] \arrow[d] & Y_1 \arrow[d, "p_1"] \\
            Z_1 \arrow[r, "i_1"] & X_1
        \end{tikzcd}      
    \end{equation}
    in $\rSchrfpSo$ which is sent to the given square (\ref{Eq:pablup_square}) under perfection. Now the condition that $p_1^{-1}(X_1 \setminus Z_1) \to X_1 \setminus Z_1$ is an isomorphism is purely topological, hence follows from the assumption that $p^{-1}(X \setminus Z) \to X \setminus Z$ is an ismorphism. 

    At this point, we have lifted any perfect abstract blow-up square to a Cartesian square in $\rSchrfpSo$ such that the corresponding Cartesian square in $\Sch_{S_1}$ is an abstract blow-up square. We further lift to an abstract blow-up square in $\SchfpSn$ via Remark \ref{Rem:Noetherian_induction}, using a descent argument similar as in the proof of Lemma \ref{Lem:Nis_square_lift}.
\end{proof}

\begin{Prop}
\label{Prop:pcdp_vs_cdp}
    A presheaf $F \in \PPP(\PerfpfpS)$ is a pcdp-sheaf if and only if $F((-)^{\pf}) \in \PPP(\SchfpSn)$ is cdp-sheaf.
\end{Prop}

\begin{proof}   
    One direction follows from the fact that $(-)^{\pf}$ preserves completely decomposed families and sends proper maps to perfectly proper maps. The other direction follows from Proposition \ref{Prop:pcdp_via_pblups} and Lemma \ref{Lem:abs_lift}.
\end{proof}

\begin{Cor}
\label{Cor:cdh_vs_pcdh}
    A presheaf $F \in \PPP(\PerfpfpS)$ is a pcdh-sheaf if and only if $F((-)^{\pf}) \in \PPP(\SchfpSn)$ is cdh-sheaf. Likewise for cosheaves.
\end{Cor}

\begin{proof}
        Combine Proposition \ref{Prop:pNis_vs_Nis} and Proposition \ref{Prop:pcdp_vs_cdp} for the first claim. For cosheaves a dual argument goes through.
\end{proof}

\subsection{Hypercompleteness}
We finish this section showing that the $\infty$-topos of $\pcdh$-sheaves over perfect schemes of finite valuative dimension has (locally) finite homotopy dimension (Theorem~\ref{Thm:idh-topology-finite-htpy-dim}). The proof follows almost verbatim Elmanto--Hoyois--Iwasa--Kelly's proof of the analogous result for the cdh-topology \cite[\S 2]{ehik--milnor-excision}, just by replacing everything by its perfection. For Krull dimensions, this is not surprising as perfection does not change the underlying topological space. But also for the valuative dimension, nothing changes (Corollary~\ref{Cor:vdim=dim-of-ZR}). For basic facts about the valuative dimension we refer to \cite[\S 2.3]{ehik--milnor-excision}.

\begin{Def}
    Let $X\in\PerfFp$ be quasi-integral\footnote{Meaning that $X$ is reduced and locally has only finitely many generic points. In this case, the subspace $X^\mathrm{gen}$ of generic points (of irreducible components) of $X$ is discerete.} of finite valuative dimension $d$. A \emph{perfect modification} of $X$ is a perfectly proper morphism $Y\to X$ in $\PerfFp$ which induces a bijection $Y^\mathrm{gen} \to[\sim] X^\mathrm{gen}$ and isomorphisms on residue fields of generic points.\footnote{Note that $Y$ is automatically reduced as perfect scheme.} Denote by $\Mdf^\pf(X)$ the full subcategory of $\Sch_X$ spanned by perfect modifications. 
    We define the \emph{perfect Zariski--Riemann space} of $X$ as the limit 
    \[ \ZR^\pf(X) \coloneqq \lim_{Y\in\Mdf^\pf(X)} Y \]
    within the category of locally ringed spaces.
\end{Def}

\begin{Rem}
\label{Rem:Mdf_vs_Mdfpf}
    Let $f \colon Y \to X$ be a perfect modification, and take a finitely presented---hence proper---model $f_0 \colon Y_0 \to X$ \cite[Cor.~3.15]{BhattProjectivity}. Then $f_0$ is a modification, hence induces an isomorphism over a dense open subset of $X$. Since $Y \to Y_0$ is a universal homeomorphism, the same is true for $f$. Since $Y$ is reduced, morphisms $Y \to Y'$ to any scheme $Y'$ separated over $X$ are determined by their restriction to any dense open subset of $Y$. In particular, $\Mdf^{\pf}(X)$ is a poset.
\end{Rem}     

Consider the functor
    \[ (-)^\pf \colon \Mdf(X) \to \Mdf^{\pf}(X). \]
By Remark \ref{Rem:Mdf_vs_Mdfpf}, this functor is surjective. It is also injective, since a morphism $Y_0 \to Y_0'$ between modifications of $X$ that is an isomorphism on perfections is an isomorphism on underlying topological spaces, hence an isomorphism in $\Mdf(X)$. Since $\Mdf(X)$ and $\Mdf^{\pf}(X)$ are posets, $(-)^{\pf}$ is an equivalence of categories. In particular, the canonical morphism
    \[ \ZR^{\pf}(X) \to \ZR(X) \]
of ringed spaces---induced by taking perfections---is a homeomorphism on underlying topological spaces.

\begin{Cor}
\label{Cor:vdim=dim-of-ZR}
    The valuative dimension of $X$ equals the Krull dimension of $\ZR^{\pf}(X)$.
\end{Cor}

\begin{proof}
    This now follows from \cite[Prop.~2.3.8]{ehik--milnor-excision}.
\end{proof}

\begin{Rem}
    We consider the site $\ZR^\pf_\mathrm{Nis}(X)$ defined as the filtered colimit of the small Nisnevich sites $Y_{\Nis}$ for all perfect modifications $Y\in\Mdf^\pf(X)$. Note that all schemes in $Y_{\Nis}$ are perfect \cite[Prop.~3.1]{BarwickTopological}. By design, we have that
    \[ \Sh_\Nis(\ZR^\pf(X)) \coloneqq \Sh(\ZR^\pf_\Nis(X)) \simeq \lim_{Y\in\Mdf^\pf(X)} \Sh_\Nis(Y). \]
    Now let $Y\in\Mdf^\pf(X)$. Note that $\Sh_\Nis(Y)$ has homotopy dimension $\leq d$ since $Y$ has finite Krull dimension $\leq d$  \cite[Thm.~3.18]{clausen-mathew--hyperdescent}. Thus the limit $\Sh_\Nis(\ZR^\pf(X))$ has homotopy dimension $\leq d$ \cite[Cor.~3.11]{clausen-mathew--hyperdescent}.
\end{Rem}

\begin{Thm}
\label{Thm:idh-topology-finite-htpy-dim}
    Let $X\in\PerfFp$ be of finite valuative dimension $d$. Then the $\infty$-topos $\Sh_\mathrm{\pcdh}(\Perfpfp_X)$ has finite homotopy dimension $\leq d$.
    Hence $\Sh_\mathrm{\idh}(\Sch_X^\fp)$ has finite homotopy dimension $\leq d$.
\end{Thm}

\begin{proof}[Proofsketch]
    This can be done using the same strategy as in \cite[\S 2.4]{ehik--milnor-excision} and replacing every object with a perfect analogue.
    
    Namely, first assume $X$ is quasi-integral, and define $\mathrm{GFin}_X^{\pfp}$ as the full subcategory of $\Perf_X^\pfp$ which is spanned by $X$-schemes whose fibres over the generic points of $X$ are finite. Then the inclusion of clean presheaves \cite[Def.~2.4.6]{ehik--milnor-excision} on $\mathrm{GFin}_X^{\pfp}$ into all presheaves admits a left-adjoint
    \[ \mathrm{cl} \colon \PPP(\mathrm{GFin}_X^{\pfp}) \to \PPP^\mathrm{cl}(\mathrm{GFin}_X^{\pfp}) \]
    which is given by the formula $Y\mapsto \colim_{Y'\to Y} F(Y')$ where $Y'$ runs over all clean perfect $Y$-schemes in $\mathrm{GFin}_X^{\pfp}$. This functor sends Nisnevich sheaves to pcdh-sheaves, cf.\ \cite[Prop.~2.4.12]{ehik--milnor-excision}.
    Thus one can define a functor $\mathrm{cl}_\mathrm{pcdh} \colon \Sh_\mathrm{pcdh}(\Perf_X^\pfp) \to \Sh_\mathrm{Nis}(\ZR^\pf(X))$ given as the composition
    \begin{align*}
        \Sh_\mathrm{pcdh}(\Perf_X^\pfp) &\to[\mathrm{res}] \Sh_\mathrm{pcdh}(\mathrm{GFin}_X^{\pfp}) \to[\mathrm{cl}] \Sh_\mathrm{pcdh}^\mathrm{cl}(\mathrm{GFin}_X^{\pfp})  \\
        &\simeq \Sh_\Nis^\mathrm{cl}(\mathrm{GFin}_X^{\pfp}) \to[\rho] \Sh_\mathrm{Nis}(\ZR^\pf(X))
    \end{align*}
    where the equivalence is the perfect analogue of \cite[Prop.~2.4.10]{ehik--milnor-excision}, and $\rho$ is given as the right-lax limit of the restriction (i.e. base change) functors along all perfect modifications of $X$, see construction on \cite[p.~1028]{ehik--milnor-excision}. Now the same proof as for \cite[Thm.~2.4.15]{ehik--milnor-excision} works, and likewise for the general case where $X$ is not assumed to be quasi-integral.
\end{proof}

\begin{Cor} \label{Cor:pcdh-topos-hypercomplete}
    Let $X \in \PerfFp$ be of finite valuative dimension. Then the $\infty$-topos $\Sh_\mathrm{pcdh}(\Perf_X^\pfp)$ is hypercomplete. In particular, the Postnikov towers in these topoi are convergent. 
\end{Cor}
\begin{proof}
    The same argument as in \cite[Cor.~2.4.16]{ehik--milnor-excision} goes through. The crucial ingredient is that an $\infty$-topos which is locally of finite homotopy dimension is hypercomplete, which is due to Lurie \cite[Prop.~4.1.5]{LurieOnInftyTopoi}.
    The last part is \cite[Prop.~7.2.1.10]{LurieHTT}.
\end{proof}

\section{The idh-topology}
\label{Sec:idh}
In this section we will describe a lift of the pcdh-topology along the perfection map $\SchfpSn \to \PerfpfpS$ such that the associated topoi become equivalent (Corollary~\ref{Cor:idh=pcdh}).  
In particular, the $\idh$-topos is hypercomplete so that we can detect equivalences on stalks which are precisely the spectra of perfect Henselian valuation rings (Proposition~\ref{Prop:points-of-idh-topology}).
Since universal homeomorphisms induce purely inseparable field extensions on residue fields, the basic idea is to replace the completely decomposed condition by the following: 
\begin{Def}
    Call a family $\{X_i \to Y\}_{i \in I}$ \emph{inseparably decomposed} if  for each $y \in Y$ there is some $i \in I$ and $x \in X_i$ over $y$ such that the induced field extension $\kappa(y) \subset \kappa(x)$ is purely inseparable. 
\end{Def}

\begin{Exm}
    Since \'{e}tale morphisms induce separable field extensions on residue fields, the inseparably decomposed \'{e}tale topology recovers the Nisnevich topology.
\end{Exm}

\begin{Def} \label{Def:idh-topology}
    The \emph{idp-topology} on $\SchfpSn$ is generated by inseparably decomposed families of proper maps. The \emph{idh-topology} on $\SchfpSn$ is generated by the idp-topology and the Nisnevich topology. 
\end{Def}

\begin{Lem}
\label{Lem:idh_to_pcdh}
    Any idh-covering family in $\SchfpSn$ is sent to a pcdh-covering family in $\PerfpfpS$ under perfection.
\end{Lem}

\begin{proof}
    Let $\{X_i \to Y\}_i$ be an inseparably decomposed family of proper maps in $\SchfpSn$. Then $\{X_i^{\pf} \to Y^{\pf}\}_i$ is a completely decomposed family of perfectly proper maps. Since perfection preserves Nisnevich covering families, the claim follows.
\end{proof}

Recall that the \emph{h-topology} on $\SchFp$ is generated by finite families of universally subtrusive maps of finite presentation.\footnote{Note that in the Noetherian case a morphism is universally subtrusive if and only if it is unversally submersive \cite[Thm.~2.8]{RydhSubmersions}. In this case the stated definition of the h-topology thus coincides with the one from \cite{VoevodskyHomology}, but this is false in general. Without the finite presentation condition, this is called the v-topology \cite{BhattProjectivity}.}

\begin{Lem}
\label{Lem:cdh_idh_h}
    The idh-topology on $\SchfpSn$ is finer than the cdh-topology and coarser than the h-topology. 
\end{Lem}

\begin{proof}
    The first point is immediate from the definitions. For the second point, since Nisnevich covers are h-covers, it suffices to show that the idp-topology is coarser than the h-topology. This is \cite[Rem.~8.3]{RydhSubmersions}.
\end{proof}

\begin{Lem}
\label{Lem:Lh_via_pf}
    For $T,X \in \SchfpSn$ it holds $\Lh(X)(T) \simeq \Hom_{\SchSn}(T^{\pf},X)$.
\end{Lem}
\begin{proof}
    This is a reformulation of \cite[Thm~8.16]{RydhSubmersions}.
\end{proof}

\begin{Lem}[cf.~{\cite[Prop.~3.2.5]{VoevodskyHomology}}] \label{Lem:idh-equivalences=universal-homeomorphisms}
A morphism $f \colon X\to Y$ in $\SchfpSn$ is a universal homeomorphism if and only if  $\mathrm{L}_{\idh}(f)$ is an equivalence in $\Sh_{\idh}(\SchfpSn)$.
\end{Lem}

\begin{proof}
    Let $f \colon X \to Y$ be a universal homeomorphism in $\SchfpSn$. Then $f$ is separated and universally closed, hence proper. Since $f$ is universally injective, it is radicial, and thus an idp-cover  \cite[\href{https://stacks.math.columbia.edu/tag/01S2}{Tag 01S2}]{stacks-project}. In particular, $\Lidh(f)$ is an effective epimorphism. 
    The diagonal $\Delta \colon X \to X \times_Y X$ is a surjective closed immersion, hence a monic idp-cover, hence an idp-equivalence. Thus $\Lidh(f)$ is also monic, and therefore an isomorphism.

    Conversely, suppose that $f$ is an idh-equivalence. By Lemma~\ref{Lem:cdh_idh_h}, $f$ is also an h-equivalence, and thus $f^{\pf}$ is an isomorphism by Corollary~\ref{Lem:Lh_via_pf}.
\end{proof}

\begin{Rem}
    From the proof of Lemma~\ref{Lem:idh-equivalences=universal-homeomorphisms}, it is clear that we can replace the idh-topology in the statement with any topology $\tau$ that contains the idp-topology and is contained in the h-topology.
\end{Rem}

Write $\rho \colon \SchfpSn \to \PerfpfpS$ for the perfection functor, and 
\[ (\rho_! \dashv \rho^*) \colon \Psh(\SchfpSn) \rightleftarrows \Psh(\PerfpfpS) \]
for the adjunction induced by precomposition with $\rho$ and left Kan extension. 
By Lemma~\ref{Lem:idh_to_pcdh} and since $\rho$ preserves finite limits, we obtain morphisms of sites
\[ (\SchfpSn,\cdh) \xleftarrow{\beta} (\SchfpSn,\idh) \xleftarrow{\alpha} (\PerfpfpS,\pcdh) \]
in the sense of \cite[\href{https://stacks.math.columbia.edu/tag/00X0}{Tag 00X0}]{stacks-project}, and thus also adjunctions
\begin{center}
    \begin{tikzcd}
        \Sh_{\cdh}(\SchfpSn) \arrow[r, shift left, "\beta^{-1}"] & \Sh_{\idh}(\SchfpSn) \arrow[l, shift left, hook, "\beta_*"] \arrow[r, shift left, "\alpha^{-1}"] & \Sh_{\pcdh}(\PerfpfpS) \arrow[l, shift left, hook, "\alpha_*"]
    \end{tikzcd}
\end{center}
via Kan extensions and localizations in the familiar way.
Note that $\beta_*$ is the natural inclusion, that $\beta^{-1}$ is the restriction of the localization functor $\Lidh \colon \Psh(\SchfpSn) \to \Sh_{\idh}(\SchfpSn)$, and that $\alpha_*$ is the restriction of $\rho^*$. Since $\rho$ is a localization by Proposition~\ref{Prop:perfection_localization}, $\alpha_*$ is fully faithful.

\begin{Prop}
\label{Prop:cdh_to_idh}
    Let $F$ be a cdh-sheaf on $\SchfpSn$. Then the following are equivalent:
    \begin{enumerate}
        \item $F$ is an idh-sheaf,
        \item $F$ inverts universal homeomorphism, and   
        \item there is a pcdh-sheaf $\hat{F}$ on $\PerfpfpS$ such that $\beta_*\alpha_*\hat{F} \simeq F$.  
    \end{enumerate}
    Moreover, the pcdh-sheaf $\hat{F}$ in the third point is the left Kan extension of $F$ along $\rho$.
\end{Prop}
\begin{proof}
    The implication from (3) to (1) follows from the fact that the essential image of $\beta_*$ consists of idh-sheaves. The implication from (1) to (2) follows from Lemma~\ref{Lem:idh-equivalences=universal-homeomorphisms}.
    
    Now suppose (2). Take a presheaf $\hat{F}$ on $\PerfpfpS$ such that $\hat{F} \rho \simeq F$ via Proposition~\ref{Prop:perfection_localization}. Then $\hat{F}$ is a $\pcdh$-sheaf by Corollary~\ref{Cor:cdh_vs_pcdh}. This implies (3). Also, observe that $\hat{F} \simeq \rho_!\rho^* \hat{F} \simeq \rho_! F$ by assumption on $\hat{F}$ and fully faithfulness of $\rho^*$.
\end{proof}

\begin{Cor}
\label{Cor:idh=pcdh}
    The adjunction
    \begin{center}
    \begin{tikzcd}
       \Sh_{\idh}(\SchfpSn) \arrow[r, shift left, "\alpha^{-1}"] & \Sh_{\pcdh}(\PerfpfpS) \arrow[l, shift left, hook, "\alpha_*"]
    \end{tikzcd}
\end{center}
is an adjoint equivalence. 
Consequently, we have a commutative diagram
    \begin{center}
        \begin{tikzcd}
            \Psh(\SchfpSn) \arrow[d, "\rho_!"] \arrow[r, "\Lcdh"] & \Sh_{\cdh}(\SchfpSn) \arrow[d, "(\beta\alpha)^{-1}"] \\
            \Psh(\PerfpfpS) \arrow[r, "\Lpcdh"] & \Sh_{\pcdh}(\PerfpfpS).
        \end{tikzcd}
    \end{center}
\end{Cor}
\begin{proof}
    We already know that $\alpha_*$ is fully faithful. By Proposition~\ref{Prop:cdh_to_idh} it is also essentially surjective. Hence the first claim follows. The second follows by passing to right adjoints.
\end{proof}

\begin{Cor} \label{Cor:idh-topos-hypercomlete}
Assume that $S_0$ is of finite valuative dimension. The the $\infty$-topos $\Sh_\idh(\SchfpSn)$ is hypercomplete. In particular, the Postnikov towers in these topoi are convergent. 
\end{Cor}
\begin{proof}
    Combine Corollary~\ref{Cor:idh=pcdh} and Corollary~\ref{Cor:pcdh-topos-hypercomplete}.
\end{proof}

\begin{Cor} \label{Cor:pcdh-subcanonical}
    The pcdh-topology on $\PerfpfpS$ is subcanonical.
\end{Cor}
\begin{proof}
    Let $X \in \PerfpfpS$ and take $X_0 \in \Schfp_S$ such that $X_0^{\pf} = X$. Then for $T_0 \in \Schfp_S$ with perfection $T$, it holds that
    \[ \Hom_{\PerfpfpS}(T,X) \simeq \Hom_{\Sch_{S}}(T,X_0) \simeq \Lh(X_0)(T_0), \]
    where the last equivalence is Lemma~\ref{Lem:Lh_via_pf}. Since $\Lh(X_0)$ is an h-sheaf, it is an idh-sheaf, and so $\Hom_{\PerfpfpS}(-,X)$ is a pcdh-sheaf by Proposition~\ref{Prop:cdh_to_idh}.
\end{proof}

\begin{Prop}
\label{Prop:points-of-idh-topology}
    The idh-points of $\SchfpSn$ are precisely the spectra of perfect Henselian valuation rings.
\end{Prop}
\begin{proof}
    By definition, an idh-point is a $S_0$-scheme $P$ such that for any morphism $P\to X$ into an object $X\in\SchfpSn$ and any idh-covering family $\{Y_i\to X\}_i$ there exists a lift $P\to Y_i$ for some $i$ \cite[\S 2]{goodwillie-lichtenbaum}. In particular, an idh-point lifts along every cdh-covering family, so that $P\simeq\Spec(V)$ for a Henselian valuation ring \cite[Thm.~2.6]{GabberPoints}. Given an element $v\in V$ we consider the associated morphism $P\to\A^1_{S_0}$ and the relative Frobenius $\A^1_{S_0}\to\A^1_{S_0}$ which is an idh-cover in $\SchfpSn$. Now the existence of a lift tells us that the element $v$ has a $p$-th root in $V$, hence $V$ is perfect.

    Conversely, let $P=\Spec(V)$ for a perfect Henselian valuation ring $V$ with fraction field $K$ be given. Since $V$ is a Henselian local ring, $P$ lifts along every Nisnevich cover. Now let $(Y_i\to X)_i$ be an inseparably decomposed family of proper morphisms, so that  $(Y_i^\pf\to X^\pf)_i$ is a pcdh-covering family. A given morphism $P\to X$ factors as $P\to X^\pf$, and we let $x\in X^\pf$ be the image of the generic point of $P$. There exists an $i$ such that the map $\Spec(K)\xrightarrow{x} X^\pf$ lifts to a morphism $\Spec(K)\to Y_i^\pf$. By the valuative criterion of properness, we get an induced lift $P\to Y_i$ (which further lifts along $Y_i^\pf\to Y).$
\end{proof}

\begin{Cor} \label{Cor:equivalences-on-stalks}
    Let $S_0\in\SchFp$ be of finite valuative dimension.
    A morphism $E\to F$ in $\Sh_\textup{idh}(\Sch^\fp_{S_0})$ is an equivalence if and only if for every morphism $x\colon \Spec(V)\to S_0$ where $V$ is a perfect Henselian valuation ring the induced map $E_x\to F_x$ is an equivalence of spaces.
\end{Cor}
\begin{proof}
By hypercompleteness, we can test equivalences on homotopy group sheaves. Then the result follows since we have enough points by a result of Gabber--Kelly \cite[Thm.~0.2]{GabberPoints}.
\end{proof}

\begin{Rem}
    As mentioned in the introduction, the idh-topology is a well-defined notion over any base which recovers the cdh-topology in characteristic zero. We expect analogues of the results of the preceding subsection to go through in general, after replacing perfect schemes with absolutely weakly normal schemes. 
\end{Rem}

\section{The perfect motivic homotopy category}
\label{Sec:perfect-motivic-cat}
Recall that $S_0$ is an $\Fp$-scheme with perfection $S$. We define a perfect version $\Hpf(S)$ of the Morel--Voevodsky motivic homotopy category (Definition~\ref{Def:perfect-H}) and a stable version $\SHpf(S)$ (Definition~\ref{Def:perfect-SH}). To relate with the classical theory there is an adjunction (Proposition~\ref{Prop:Phi_comm_susp})
\[ \Phi_{S_0} \dashv \Psi_{S_0} \colon \SH(S_0) \rightleftarrows \SHpf(S) \]
where the right adjoint $\Psi_{S_0}$ is fully faithful (Proposition~\ref{Prop:Psi_ff}). From the latter we deduce that the assignment $S_0\mapsto\SHpf(S)$ is the initial coefficient system which inverts universal homeomorphisms (Corollary~\ref{Cor:SH_motivic_coefficient_systems}).

\subsection{The unstable perfect motivic homotopy category}
Let $\overline{\G}_m$ (resp.\ $\overline{\A}{}^n$, resp.\  $\overline{\P}{}^n$) be the perfection of $\G_{m,\F_p}$ (resp.\ of $\A^n_{\F_p}$, resp.\ of $\P^n_{\F_p}$). By adjunction, these are perfect versions of the multiplicative group (resp.\ of affine space, resp.\ of projective space), in the sense that the functor of points description remains the same.
\begin{Def}[{cf.~\cite[Def.~A.25]{ZhuAffine}}]
    A morphism $f \colon X \to Y$ of perfect schemes is \emph{perfectly smooth} if it is perfectly finitely presented and for all $x \in X$ there are open neighborhoods $U$ of $x$ and $V$ of $f(x)$ such that $f(U) \subset V$, and such that the restriction of $f$ factors as 
    \[ U \to \overline{\A}{}^n \times V \to V \]
    where the first map is \'{e}tale and the second map is the projection.
\end{Def}
Write $\Sm_{S}^{\pf}$ for the full subcategory of $\Perf_S$ spanned by perfectly smooth schemes over $S$. 

\begin{Def}
\label{Def:perfect-H}
    A \emph{perfect motivic space} over $S$ is an $\overline{\A}{}^1$-homotopy invariant Nisnevich sheaf  on $\Sm_S^{\pf}$. The \emph{perfect motivic homotopy category} $\H^{\pf}(S)$ is the full subcategory of $\PPP(\Sm_S^{\pf})$ spanned by perfect motivic spaces over $S$.
\end{Def}
Endow $\H^{\pf}(S)$ with the Cartesian monoidal structure. 

\begin{Prop}\label{Prop:FuncBCFandProjF}
    We have a functor 
    \begin{align*}
        \Hpf \colon \PerfFp^{\op} \to \PrLo 
    \end{align*}
    such that $f^* \coloneqq \Hpf(f)$ admits a left adjoint $f_\sharp$ for any perfectly smooth map $f$. Moreover, the left adjoints $(-)_\sharp$ satisfy base change and the projection formulas. 
\end{Prop}

\begin{proof}
    The same strategy as in \cite[\S 4.1]{HoyoisSix} goes through.\footnote{To lift the functor $S \mapsto \H^{\pf}(S)$ along the forgetful functor $\PrLo \to \PrL$, use \cite[Prop.~2.2.1.9]{LurieHA}.} For the existence of the adjunctions $f^* \dashv f_*$, use Lemma \ref{Lem:pfNis_excission}.
\end{proof}

Define the \emph{perfect motivic localization}
\[ \L^{\pf}_S \colon \PPP(\Sm^\pf_S) \to \H^\pf(S) \]
as the left adjoint of the inclusion $\H^{\pf}(S) \to \PPP(\Sm^\pf_S)$. Also write the composition $\Lpf_Sh_S$ with the Yoneda embedding $h_S$ simply as $\Lpf_S$.\footnote{We may omit $S$ from the notation when the base is clear.}

\begin{Prop}
\label{Prop:Hpf_gend_by_reps}
    The $\infty$-category $\H^\pf(S)$ is generated under sifted colimits by objects of the form $\Lpf_SX$, where $X$ is the perfection of a smooth, affine $S_0$-scheme $X_0$ such that the structure map $X_0 \to S_0$ factors through an \'{e}tale map $X_0 \to \A^n_{S_0}$.
\end{Prop}
\begin{proof}
    By the same argument as the proof of \cite[Prop.~3.16]{HoyoisSix}, $\H^{\pf}(S)$ is generated under sifted colimits by representables $\Lpf_S(X)$. Now the condition that $X$ be of the desired form holds Zariski-locally by definition, whence the claim.
\end{proof}

\subsection{The pointed category}
\label{subsec:pointed-H-pf}
In this section we pass from the unpointed category $\H^\pf(S)$ to the pointed category $\H^\pf_\bullet(S)$. We follow the same general outline as in \cite[\S 5.1]{HoyoisSix} to lift the results on perfect motivic spaces to pointed perfect motivic spaces. 
The universal way of pointing an $\infty$-category with a final object leads us to the following:

\begin{Def}
Let the $\infty$-category of \emph{pointed perfect motivic spaces} over $S$ be defined as 
\[\H_{\bullet}^\pf(S)\coloneqq \H^\pf(S)_{S/}.\]
By  $(-)_+\colon \H^\pf(S) \to  \H_{\bullet}^\pf(S)$ we denote the left adjoint to the forgetful functor. 
\end{Def}

Using \cite[Prop.~4.8.2.11]{LurieHA}
we get a unique closed symmetric monoidal structure on $\H_{\bullet}^\pf(S)$ that is compatible with the monoidal structure on $\H^\pf(S)$. We will denote its tensor product by $\wedge$.

For any morphism $f \colon T \to S$ of perfect schemes it holds that $f^*$ and $f_*$ on $\Hpf$ preserve final objects, and thus induce an adjunction
\[ (f^* \dashv f_*) \colon \Hpf_\bullet(S) \rightleftarrows \Hpf_\bullet(T). \]
If $f$ is perfectly smooth then $f^*$ hence has a left adjoint $f_\sharp \colon \H_\bullet(T) \to \H_\bullet(S)$ such that $f_\sharp(X_+) \simeq (f_\sharp X)_+$ for all $X \in \Hpf(T)$.

\begin{Cor}
    We obtain a functor
    \[ \Hpf_\bullet \colon \PerfFp \to \PrLo  \]
    such that $f^*$ admits a left adjoint $f_\sharp$ for any perfectly smooth map $f$. Moreover, the left adjoints $(-)_\sharp$ satisfy base change and the projection formulas. 
\end{Cor}
\begin{proof}
    This follows from the unpointed case.
\end{proof}

\subsection{The stable category}
Let $\barT \in \H_\bullet(S)$ be the perfect motivic localization of the pointed object $(\barP \times S,\infty)$. Write $\Sigma_{\barT} \coloneqq \barT \wedge (-)$ and $\Omega_{\barT} \coloneqq \Homu_{\H_\bullet}(\barT,-)$, so that $\Sigma_{\barT} \dashv \Omega_{\barT}$.
\begin{Def} \label{Def:perfect-SH}
For a perfect scheme $S$ the  \emph{stable perfect motivic homotopy category} is the $\infty$-category
\[ \SH^\pf(S) \coloneqq \H^\pf_\bullet(S)[\barT^{-1}]. \]
\end{Def}

\begin{Rem}
\label{Rem:SHpf-as-barT-spectra}
    The $\infty$-category $\SH^\pf(S)$ is defined by formally inverting the object $\barT$ with respect to the tensor product; see \cite[Def.~2.6]{robalo-bridge} for an explicit model of this $\infty$-category. The functor $\H^\pf_\bullet(S) \to \SH^\pf(S)$ is the unique symmetric monoidal functor into a symmetric monoidal $\infty$-category where the object $\barT$ becomes $\otimes$-invertible \cite[Prop.~2.9]{robalo-bridge}.  The cyclic permutations on $\barP\otimes\barP\otimes\barP$ are $\barA$-equivalent to the identity which follows by functoriality from the unperfect case \cite[Lem.~4.4]{voevodsky-icm-1998}. Hence we can describe $\SH^\pf(S)$ also as the $\infty$-category of $\barT$-spectra, computed as
    \[ \colim\bigl( \H^\pf_\bullet(S) \xrightarrow{\Sigma_{\barT}} \H^\pf_\bullet(S) \xrightarrow{\Sigma_{\barT}} \ldots \bigr) \simeq \lim\bigl( \H^\pf_\bullet(S) \xleftarrow{\Omega_{\barT}} \H^\pf_\bullet(S) \xleftarrow{\Omega_{\barT}} \ldots \bigr), \]
    where the colimit is computed in $\PrL$, and the limit in $\PrR$ hence in $\Cat_\infty$.\footnote{Recall that limits in $\PrL$ are computed on the level of underlying $\infty$-categories, likewise for limits in $\PrR$ \cite[Prop.~5.5.3.13,~Thm.~5.5.3.18]{LurieHTT}.}
    
    The description of $\SH^\pf(S)$ in terms of $\barT$-spectra leads to the following alternative description. Write $\Sp^\N(\H_\bullet^\pf(S),\barT)$ for the $\infty$-category of \emph{(perfect motivic) prespectra}, i.e., sequences $(X_i)_{i \in \N}$ in $\H_\bullet^\pf(S)$ together with bonding maps $\barT \wedge X_i \to X_{i+1}$. Then $\SH^\pf(S)$ is the full subcategory of $\Sp^\N(\H_\bullet^\pf(S),\barT)$ spanned by those perfect prespectra $(X_i)_{i \in \N}$ for which the transposed maps $X_i \to \Omega_{\barT} X_{i+1}$ are all equivalences. We have a left adjoint
    \[ \L_{\st} \colon \Sp^\N(\H_\bullet^\pf(S),\barT) \to \SH^\pf(S) \]
    to the inclusion $\SH^\pf(S) \subset \Sp^\N(\H_\bullet^\pf(S),\barT)$ \cite[\S 2C]{BachmannNotes}.
\end{Rem}

\vspace{12pt} 
We write 
\[ (\Sigma^\infty_{\barT} \dashv \Omega^\infty_{\barT}) \colon \H_\bullet^\pf(S) \rightleftarrows \SH^\pf(S) \]
for the adjunction obtained by inverting $\barT$. By the universal property of inverting $\barT$, for any morphism of perfect schemes $f \colon T \to S$ we obtain an adjunction
\[ (f^* \dashv f_*) \colon \SHpf(S) \rightleftarrows \SHpf(T) \]
such that $f^*$ is symmetric monoidal and with $f^* \Sigma^\infty_{\barT} \simeq \Sigma^\infty_{\barT} f^*$. By comparing universal properties we see that the functor
\[ \H^{\pf}_\bullet(T)\otimes_{\H^{\pf}_\bullet(S)}\SH^{\pf}(S)\to\SH^{\pf}(T)\]
induced by $f^*:\SH^{\pf}(S)\to \SH^{\pf}(T)$ is an equivalence.\footnote{Here, $\otimes$ is the Lurie tensor product on $\PrL$.} In particular, if $f$ is perfectly smooth we obtain an adjunction
\[ (f_\sharp \dashv f^*) \colon \SHpf(T) \rightleftarrows \SHpf(S) \]
by base change along $\Hpf_\bullet(S) \to \SHpf(S)$.

\begin{Prop} \label{Prop:perfect-SH-functor}
    We have a functor
    \[ \SH^\pf \colon \PerfFp^\op \to \PrLotimesst \]
such that for every perfectly smooth morphism $f$ the functor $f^*$ has a left-adjoint $f_\sharp$ which satisfies base change and the projection formula. 
\end{Prop}
\begin{proof}
    The universal property of inverting $\barT$ yields a functor $\SH^\pf \colon \PerfFp^\op \to \PrL_\st$ which is pointwise symmetric monoidal.
    In order to enrich this functor to a functor with values in $\PrLotimesst$ one can proceed similarly as in Robalo's thesis \cite[\S 9.1]{robalo-these}.    

   Perfectly smooth base change follows from perfectly smooth base change in the pointed case, together with the universal property of the stabilization $\H_\bullet^{\pf} \to \SH^{\pf}$.
   The projection formula for $\sharp$-direct images follows by similar reasoning as in \cite[Prop.~1.23]{KhanVoevodsky}.    
\end{proof}

\section{Perfect motivic homotopy theory}
\label{Sec:perfect-motivic}

\subsection{Adjunctions between the motivic and the perfect motivic}
\label{subsec:adjunctions}
Write
\[ \varphi^{\PPP} \dashv \psi^{\PPP} \colon \PPP(\Sm_{S_0}) \rightleftarrows \PPP(\Sm_S^{\pf}) \]
for the adjunction where $\psi^{\PPP}$ is given by precomposition with the perfection map $\Sm_{S_0} \to \Sm_S^{\pf}$ and $\varphi^{\PPP}$ by left Kan extension.

\begin{Rem}
    By adjunction and the formula for left Kan extensions, it holds that
    \[ \varphi^{\PPP}(F)(Y) \simeq \colim_{Y \to Y_\alpha} F(Y_\alpha) \]
    where the indexing category is the slice $Y/\Sm_{S_0}$.
\end{Rem}

\begin{Lem}
\label{Lem:phiP_yoneda}
    For $X_0 \in \Sm_{S_0}$ with perfection $X$, it holds that $\varphi^{\PPP}(h(X_0)) \simeq h(X)$. It follows that $\varphi^{\PPP}$ is symmetric monoidal (for the Cartesian monoidal structure).
\end{Lem} 

\begin{proof}
    This follows from the fact that $\varphi^\PPP$ is the Yoneda extension.
\end{proof}

By Proposition~\ref{Prop:pNis_vs_Nis}, the functor $\psi^{\PPP}$ restricts to perfect Nisnevich sheaves. Clearly, it sends $\barA$-invariant presheaves to $\A^1$-invariant presheaves. We thus obtain an adjunction
\[ \varphi \dashv \psi \colon \H(S_0) \rightleftarrows \H^{\pf}(S) \]
such that $\psi$ is the restriction of $\psi^{\PPP}$, and $\varphi$ is the composition of the restriction of $\varphi^\PPP$ to $\H(S_0)$ with the localization $\PPP(\Sm_{S}^{\pf}) \to \H^{\pf}(S)$. By Lemma~\ref{Lem:phiP_yoneda} and construction, both $\varphi
$ and $\psi$ preserve final objects. We thus obtain an adjunction on pointed objects, also written $\varphi \dashv \psi$. 

\begin{Prop}\label{Prop:Phi_comm_susp}
    There is a natural transformation
    \[ \Phi \colon \SH(-) \to \SH^{\pf}((-)^\pf) \]
    of functors $\SchFp^{\op} \to \PrLotimesst$, with pointwise right adjoint $\Psi$ on underlying $\infty$-categories, such that the diagrams
    \begin{center}
    \begin{tikzcd}
        \H_\bullet(S_0) \arrow[r, "\Sigma^\infty_{\T}"] \arrow[d, "\varphi"] & \SH(S_0) \arrow[d, "\Phi"] && \H_\bullet(S_0) & \SH(S_0) \arrow[l, "\Omega_{\T}^\infty"] \\
        \H^{\pf}_\bullet(S) \arrow[r, "\Sigma_{\overline{\T}}^\infty"] & \SH^{\pf}(S) && \H^{\pf}_\bullet(S) \arrow[u, "\psi"] & \SH^{\pf}(S) \arrow[u, "\Psi"] \arrow[l, "\Omega_{\overline{\T}}^{\infty}"]
    \end{tikzcd}
    \end{center}
    commute for every $S_0 \in \SchFp$ with perfection $S$, and such that $\Phi$ commutes with $\sharp$-direct images of smooth morphisms in the obvious way.
\end{Prop}

\begin{proof}
    The existence of $\Phi \colon \SH(S_0) \to \SH(S)$ follows from the fact that $\varphi$ is symmetric monoidal and sends $X_0 \in \Sm_{S_0}$ to the perfection $X \in \Sm_S^{\pf}$. The functor $\psi^\PPP$ is clearly natural in $S_0$, hence so is $\psi$, hence so is $\varphi$. Now $\Phi$ is obtained from $\varphi$ by $\otimes$-inverting $\T \mapsto \barT$. Hence naturality follows by similar arguments as in \cite[\S 9.1]{robalo-these}. 
\end{proof}

\begin{Not}
    To emphasize the base of these constructions, we may write $\varphi,\Phi,\dots$ as $\varphi_{S_0},\Phi_{S_0},\dots$. 
\end{Not}

\begin{Rem}
\label{Rem:phi_is_phipf}
    Let $S \to[\rho] S_0 \to[f] S_0'$ be universal homeomorphisms with $S$ being perfect. The diagram
    \begin{center}
        \begin{tikzcd}
            \SH(S_0') \arrow[d, "f^*", swap] \arrow[dr, "\Phi_{S_0'}"] \\
            \SH(S_0) \arrow[r,swap, "\Phi_{S_0}"] & \SH^{\pf}(S)
        \end{tikzcd}
    \end{center}
    commutes by naturality of $\Phi$.
\end{Rem}

\subsection{$\Perf$-fibered perfect motivic spaces and spectra}
\label{Subsec:fibred-spaces-spectra}
We use the theory of $\Sch$-fibered spaces and spectra from \cite{Khan-SH_cdh}: we write $\uH(S_0)$ for the symmetric monoidal $\infty$-category of $\A^1$-invariant Nisnevich sheaves on $\SchfpSn$, with associated pointed category $\uH_\bullet(S_0)$, and put $\uSH(S_0) \coloneqq \uH_\bullet(S_0)[\T^{-1}]$. The functor $\iota^* \colon \uSH(S_0) \to \SH(S_0)$ obtained by restriction along the inclusion $\iota \colon \Sm_{S_0} \to \SchfpSn$ admits fully faithful left and right adjoints $\iota_! \dashv \iota^* \dashv \iota_*$ obtained by Kan extensions (and localization in the case of $\iota_!$).\footnote{In \cite{Khan-SH_cdh} the left adjoint is written $\L\iota_!$ but we omit the localization from the notation.} We will now give a perfect analogue of this picture.

Let $\uHpf(S)$ be the symmetric monoidal $\infty$-category of $\barA$-invariant Nisnevich sheaves on $\PerfpfpS$. Write $\uHpf_\bullet(S) \coloneqq \uHpf(S)_{S/}$ for the pointed $\infty$-category and $\uSHpf(S)$ for the associated $\infty$-category of $\barT$-spectra in $\uHpf_\bullet(S)$. We write
\[ (\Sigma^\infty_{\barT} \dashv \Omega^\infty_{\barT}) \colon \uHpf_\bullet(S) \rightleftarrows \uSHpf(S) \]
for the associated adjunction.

Write also $\iota \colon \Sm_S^{\pf} \to \PerfpfpS$ for the inclusion. By similar arguments as in the proofs of \cite[Lem.~1, Prop.~2]{Khan-SH_cdh}, we obtain adjunctions $\iota_! \dashv \iota^* \dashv \iota_*$ of the form
\begin{center}
\begin{tikzcd}
\Hpf(S) \arrow[r, bend left, "\iota_!"] \arrow[r, bend right, "\iota_*", swap] & \uHpf(S), \arrow[l, "\iota^*", swap]
\end{tikzcd}
\end{center}
such that $\iota^*$ is precomposition with $\iota$, $\iota_*$ is right Kan extension, and $\iota_!$ is induced by left Kan extension and localization. Since $\iota$ is fully faithful, so is $\iota_*$, hence so is $\iota_!$. These functors stabilize to adjunctions 
\begin{center}
\begin{tikzcd}
\SHpf(S) \arrow[r, bend left, "\iota_!"] \arrow[r, bend right, "\iota_*", swap] & \uSHpf(S), \arrow[l, "\iota^*", swap]
\end{tikzcd} 
\end{center}
such that $\iota^*$ and $\iota_*$ commute with $\Omega^\infty_{\barT}$ and $\iota_!$ commutes with $\Sigma^\infty_{\barT}$. All the basic functionality of these constructions mentioned in \cite[\S 1.7]{Khan-SH_cdh} go through after replacing the necessary terms with their perfect analogues. 

Write $\rho \colon \SchfpSn \to \PerfpfpS$ for the perfection functor. We have an adjunction
\[ (\uphi \dashv \upsi) \colon \uH(S_0) \rightleftarrows \uHpf(S) \]
where $\upsi$ is precomposition with $\rho$ and $\uphi$ is induced by left Kan extension and localization. As before, we obtain an associated adjuntion $\uphi \dashv \upsi$ on pointed objects, which stabilizes to an adjunction
\[ (\uPhi \dashv \uPsi) \colon \uSH(S_0) \rightleftarrows \uSHpf(S) \]
such that $\uPhi \Sigma^\infty_{\T} \simeq \Sigma^\infty_{\barT} \uphi$, and such that $\uPhi f_0^* \simeq f^* \uPhi$ for any morphism $f_0 \colon T_0 \to S_0$ of $\Fp$-schemes with perfection $f$.

As before, we may add a subscript and write $\uphi_{S_0},\uPhi_{S_0},...$ in order to emphasize the dependency on $S_0$.

\begin{Not}
    Let $\Sm_S^{\pf '}$ be the essential image of the perfection functor $\Sm_S \to \Sm_S^{\pf}$. By Zariski descent, $\Sm_S^{\pf}$ can be replaced by $\Sm_S^{\pf '}$ in the construction of $\H^{\pf}(S)$, which we do without mention when convenient.
\end{Not}

\begin{Lem}
\label{Lem:phi_iota}
    The diagram
    \begin{center}
        \begin{tikzcd}
            \uH(S) \arrow[r, "\iota^*"] \arrow[d, "\uphi"] & \H(S) \arrow[d, "\varphi"] \\
            \uHpf(S) \arrow[r, "\iota^*"] & \Hpf(S)
        \end{tikzcd}
    \end{center}
    commutes via the natural map $\varphi \iota^* \to \iota^* \uphi$.
\end{Lem}

\begin{proof}
    We first do the corresponding statement on presheaves. Let $X \in \Sm_S^{\pf'}$ and consider the functor
    \[ G \colon X / \Sm_S \to X / \Sch_S^{\fp} \]
    induced by inclusion. Take a smooth model $X_0$ of $X$, and let $X_n$ be the scheme over $S$ induced by the $n$-fold iterated Frobenius $X_n \to X_{n-1} \to \dots \to X_0 \to S$. Note that $X_n$ is smooth over $S$ by naturality of the Frobenius and perfectness of $S$. Let $W \in X/\Sch_S^{\fp}$ be given. Since $W$ is finitely presented over $S$ it holds 
    \[ \Hom_S(X,W) \simeq \colim_n \Hom_S(X_n,W), \]
    and the map $X \to W$ factors through some $X_n$, and likewise for any $Y \in G/W$ in the comma category. For $Y \in G/W$ let $n_Y$ be the smallest such $n$. Since for $Y\to Y'$ in $G/W$ it holds $n_Y \geq n_{Y'}$, we get an endofunctor on $G/W$ given by $Y \mapsto X_{n_Y}$, furthermore the maps $X_{n_Y} \to Y$ constitute a natural transformation from this endofunctor to the identity on $G/W$. Hence $G/W$ deformation retracts to a filtered subcategory, and in particular is contractible. Therefore $G$ is initial by Quillen's Theorem A. It follows that the natural map $\varphi^{\PPP} \iota^* \to \iota^* \underline{\varphi}^\PPP$ is an equivalence from the pointwise formula for left Kan extensions. 

    Since $\varphi^\PPP$ commutes with $\iota^*$, it holds that $\psi^\PPP$ commutes with $\iota_*$, hence that $\psi$ commutes with $\iota_*$, hence that $\varphi$ commutes with $\iota^*$.\footnote{Here and in what follows, we abuse some notation and say that $\varphi^\PPP$ commutes with $\iota^*$ to mean that $\varphi^{\PPP} \iota^* \to \iota^* \underline{\varphi}^\PPP$ is an equivalence. Likewise for the other commutativity statements.}
\end{proof}

\begin{Cor}
\label{Cor:psi_ff}
    The functor $\psi_{S} \colon \H^{\pf}_\bullet(S) \to \H_\bullet(S)$ is fully faithful.
\end{Cor}

\begin{proof}
    Since $\psi$ and $\iota_*$ preserve terminal objects it holds by Lemma~\ref{Lem:phi_iota} that $\psi$ also commutes with the fully faithful functor $\iota_*$ in the pointed case.  It thus suffices to show that the functor
    $ \upsi \colon \uH_\bullet^{\pf}(S) \to \uH_\bullet(S) $
    is fully faithful, which follows from Proposition~\ref{Prop:perfection_localization}.
\end{proof}

\begin{Lem}
\label{Lem:SHpf_via_basechange}
    The natural map
    \[ \H_\bullet^{\pf}(S) \otimes_{\H_\bullet(S_0)} \SH(S_0) \to \SH^{\pf}(S) \]
    is an equivalence. Consequently, the functor $\Phi_{S_0}$ is obtained from $\varphi_{S_0}$ by base change along $\H_\bullet(S_0) \to \SH(S_0)$.
\end{Lem}

\begin{proof}
    Since $\SH(S) \simeq \H_\bullet(S) \otimes_{\H_\bullet(S_0)} \SH(S_0)$ it suffices to do the case $S=S_0$. Now this follows from comparing universal properties.
\end{proof}

\begin{Prop}
\label{Prop:Psi_ff}
    The functor $\Psi_{S_0} \colon \SHpf(S) \to \SH(S_0)$ is fully faithful.
\end{Prop}

\begin{proof}
    Let $\rho \colon S \to S_0$ be the perfection. Recall that we have a factorization
    \[ \Psi_{S_0} \simeq \rho_* \Psi_S \colon \SH^{\pf}(S) \to \SH(S) \to \SH(S_0).\]
    Now $\rho_*$ is fully faithful \cite[Prop.~2.2.1]{Elmanto-Khan-perfection}, hence we may assume without loss of generality that $S=S_0$. This case follows from \cite[Prop.~1.3.14, Lem.~1.5.4]{annala-iwasa-motivic-spectra} by Corollary~\ref{Cor:psi_ff}.
\end{proof}

\begin{Rem}
\label{Rem:prespectra_localizations}
    The key in the proof of Proposition~\ref{Prop:Psi_ff} is \cite[Lem.~1.5.4]{annala-iwasa-motivic-spectra} which is a general result on the preservation of localizations in the passage to spectra. We sketch a model-theoretic argument for the interested reader.

    Let $f \colon \CCC \to \CCC'$ be a symmetric mononoidal left adjoint functor between presentably symmetric monoidal $\infty$-categories, with fully faithful right adjoint $g$. Let $T \in \CCC$ and $T' \coloneqq f(T) \in \CCC'$ be symmetric objects.  Write $\Sigma_T \coloneqq T \otimes (-)$ for the endofunctor with right adjoint $\Omega_T \coloneqq \Homu_\CCC(T,-)$, and write $\Sp^\N(\CCC,T)$ for the $\infty$-category of $T$-prespectra in $\CCC$, as in \cite[\S 2C]{BachmannNotes}. Likewise for $\CCC'$. We first briefly explain how the method of prolongation from \cite[\S 5]{hovey-symmetric-spectra} induces for us an adjunction
    \[ (\Sp(f) \dashv \Sp(g)) \colon \Sp^\N(\CCC,T) \rightleftarrows \Sp^\N(\CCC',T'), \]
    where $\Sp(g)$ is fully faithful, together with a commutative diagram
    \begin{equation}
    \label{Eq:Presp_comm}
        \begin{tikzcd}
            \CCC \arrow[r] \arrow[d, "f"] & \Sp^\N(\CCC,T) \arrow[r] \arrow[d, "\Sp(f)"] & \CCC[T^{-1}] \arrow[d, "F"] \\
            \CCC' \arrow[r]  & \Sp^\N(\CCC',T') \arrow[r] & \CCC'[{T'}^{-1}],
        \end{tikzcd}
    \end{equation}
    where the horizontal arrows are the natural functors, and $F$ is obtained from $f$ by the universal property of $\otimes$-inverting $T$. 

    The existence of $\Sp(f) \dashv \Sp(g)$ follows from \cite[Lem.~5.3]{hovey-symmetric-spectra}, via the dictionary between presentably symmetric monoidal $\infty$-categories on one hand, and simplicial, combinatorial and left proper symmetric monoidal model categories on the other \cite[\S 1.5.2]{robalo-bridge}, \cite[Thm.~1.1, Thm.~2.8]{NikolausPresentably}, \cite[\S A.3.7]{LurieHTT}. As prolongation on the level of model categories, $\Sp(g)$ sends a $T'$-spectra $(Y_n)_{n \in \N}$ to the $T$-spectrum $(gY_n)_{n \in \N}$, with bonding maps the adjoint of the composition $gY_i \to g\Omega_{T'}Y_{i+1} \simeq \Omega_TgY_{i+1}$ for all $i$. Likewise, $\Sp(f)$ sends $(X_n)_{n \in \N}$  to $(fX_n)_{n \in \N}$ with obvious bonding maps. It follows that $\Sp(f) \dashv \Sp(g)$ is a Quillen reflection on the level of model categories, hence that $\Sp(g)$ is also fully faithful on the level of $\infty$-categories. The left-hand square in (\ref{Eq:Presp_comm}) commutes on the level of model categories, where the functors involved are left Quillen. It follows that the associated square on the level of $\infty$-categories also commutes.

    For the commutativity of the right-hand square in (\ref{Eq:Presp_comm}), observe that the functor $\Sp^\N(\CCC,T) \to \CCC[T^{-1}]$ is the localization functor at the stable equivalences, i.e., that  $\CCC[T^{-1}]$ is the full subcategory of $\Omega$-spectra in $\Sp^\N(\CCC,T)$, likewise for $\CCC'$ \cite[Rem.~2.3]{BachmannNotes}, \cite[Thm.~2.26]{robalo-bridge}. Since $\Sp(g)$ preserves $\Omega$-spectra, there is a natural functor $F_2 \colon \CCC[T^{-1}] \to \CCC'[{T'}^{-1}]$ that makes the right-hand square in (\ref{Eq:Presp_comm}) commute. Now the identification from \cite[Thm.~2.26]{robalo-bridge} is symmetric monoidal, hence $F_2$ is also the unique functor induced by $\otimes$-inverting $T$. Hence $F \simeq F_2$ as claimed.

    For formal reasons, the functor $F$ has a right adjoint $G$. From the commutativity (\ref{Eq:Presp_comm}), we conclude that $G$ is fully faithful, as it is the restriction of $\Sp(g)$. Taking $f \coloneqq \varphi_S$ thus recovers Proposition~\ref{Prop:Psi_ff}, again using  \cite[Prop.~2.2.1]{Elmanto-Khan-perfection} and Corollary~\ref{Cor:psi_ff}.
\end{Rem}

\subsection{Perfection of closed immersions}
For a category $\CCC$ with initial object $\varnothing \in \CCC$, denote by $\PPP_\varnothing(\CCC)$ the full subcategory of $\PPP(\CCC)$ spanned by \emph{reduced} presheaves, i.e., by those $F$ satisfying $F(\varnothing)\simeq *$. Observe that $\Lpf$ factors as
\[ \PPP(\Sm^\pf_S) \to \PPP_\varnothing(\Sm^\pf_S) \to \H^\pf(S), \]
where the first arrow is the left adjoint to the inclusion $\PPP_\varnothing(\Sm^\pf_S) \to \PPP(\Sm^\pf_S)$. Note that pushforward along any morphism restricts to reduced presheaves.

\begin{Lem}
\label{Lem:varnothing_presheaves_lower_star}
Let $i \colon Z \to S$ be a closed immersion of perfect schemes. The functor $i_{*} \colon \PPP_\varnothing(\Sm_Z^\pf) \to \PPP_\varnothing(\Sm_S^\pf)$ commutes with $\Lpf$. Consequently, the functor $i_* \colon \H_\bullet^\pf(Z) \to \H_\bullet^\pf(S)$ preserves colimits.
\end{Lem}
\begin{proof}
Note that \'{e}tale morphisms can be lifted Nisnevich-locally along $i$, and a scheme \'{e}tale over a perfect base is perfect. Hence the same argument as in the proof of \cite[Thm.~3.1.1]{KhanMorelVoevodsky} goes through for the first claim, and the same argument as in \cite[Cor.~3.1.2]{KhanMorelVoevodsky} for the second claim.
\end{proof}

\begin{Lem}
\label{Lem:push_phi}
    Let $f \colon X \to S$ be a morphism of perfect schemes. Then the diagram
    \begin{center}
        \begin{tikzcd}
            \PPP(\Sm_X) \arrow[r, "f_*"] \arrow[d, "\varphi_X^\PPP"] & \PPP(\Sm_S) \arrow[d, "\varphi_S^\PPP"] \\
            \PPP(\Sm_X^{\pf '}) \arrow[r, "f_*"]  &  \PPP(\Sm_S^{\pf '})
        \end{tikzcd}
    \end{center}
    commutes via the natural transformation $\varphi_S^{\PPP} f_* \to f_* \varphi_X^{\PPP}$. Likewise for the corresponding statement on $\PPP_\varnothing(-)$.
\end{Lem}

\begin{proof}
    Let $T \in \Sm_S^{\pf '}$. Consider the functor
    \[ G \colon T /\Sm_S \to T_X / \Sm_X \]
    on comma categories induced by pullback along $f$. By a similar argument as in the proof of Lemma~\ref{Lem:phi_iota} one shows that $G$ is initial, whence the first claim.

    For the corresponding statement on $\PPP_\varnothing(-)$, note that all functors in the given diagram restrict to reduced presheaves.
\end{proof}

\begin{Lem}
\label{Lem:Phi_commutes_i}
    Let $i_0\colon Z_0 \to S_0$ be a closed immersion of $\Fp$-schemes with perfection $i \colon Z \to S$. Then the diagram
    \begin{center}
        \begin{tikzcd}
            \SH(Z_0) \arrow[r, "{i_0}_*"] \arrow[d, "\Phi_{Z_0}"'] & \SH(S_0) \arrow[d, "\Phi_{S_0}"] \\
            \SH^{\pf}(Z) \arrow[r, "i_*"] & \SH^{\pf}(S)
        \end{tikzcd}
    \end{center}
    commutes via the natural transformation $\Phi_{S_0}{i_0}_* \to i_* \Phi_{Z_0}$.
\end{Lem}

\begin{proof}
    By base change for $\SH$ and since $\Phi_{S_0} \simeq \Phi_S \rho^*$ for $\rho \colon S \to S_0$ the perfection, we reduce to the case where $S$ and $Z$ are perfect. 

    We first do the corresponding statement on (perfect) motivic spaces: that $\varphi_S i_* \simeq i_*\varphi_Z$. Consider the diagram
    \begin{center}
        \begin{tikzcd}
            \PPP_\varnothing(\Sm_Z) \arrow[r, "\L_Z"] \arrow[d, "\varphi_Z^{\PPP}"] & \H(Z) \arrow[d, "\varphi_Z"] \arrow[r, "i_*"] & \H(S) \arrow[d, "\varphi_S"] \\
            \PPP_\varnothing(\Sm_Z^{\pf'}) \arrow[r, "\L^{\pf}_Z"] & \H^{\pf} (Z) \arrow[r, "i_*"] & \H^{\pf}(S).
        \end{tikzcd}
    \end{center}
    Since the left-hand square commutes by definition and $\L_Z$ is a localization, it suffices to show that the outer square commutes. By Lemma \ref{Lem:varnothing_presheaves_lower_star} (and its imperfect analogue), this outer square is equivalent to the outer square of the diagram
    \begin{center}
        \begin{tikzcd}
            \PPP_\varnothing(\Sm_Z) \arrow[r, "i_*"] \arrow[d, "\varphi_Z^{\PPP}"] & \PPP_\varnothing(\Sm_S) \arrow[d, "\varphi_S^{\PPP}"] \arrow[r, "\L_S"] & \H(S) \arrow[d, "\varphi_S"] \\
            \PPP_\varnothing(\Sm_Z^{\pf'}) \arrow[r, "i_*"] & \PPP_\varnothing(\Sm_S^{\pf'}) \arrow[r, "\L_S^{\pf}"] & \H^{\pf}(S).
        \end{tikzcd}
    \end{center}
    In this diagram, the left-hand square commutes by Lemma \ref{Lem:push_phi}, and the right-hand square commutes as before. Whence the claim.

    The pointed case follows from the unpointed case since the functors in question preserve final objects. Now the statement on (perfect) motivic spectra follows from the statement on pointed (perfect) motivic spaces, since the diagram on spectra is obtained from the corresponding diagram on pointed spaces via base change along $\H_\bullet(S) \to \SH(S)$ by Lemma~\ref{Lem:SHpf_via_basechange}.
\end{proof}

\subsection{Perfect Thom spaces}
Let $p_0 \colon E_0 \to S$ be a vector bundle with zero section $s_0 \colon S \to E_0$. Write $p \colon E \to S$ and $s \colon S \to E$ for their perfections.

\begin{Not}
    Define the \emph{perfect Thom twist} as
    \[ \Sigma^{E} \coloneqq p_\sharp s_* \colon \H^{\pf}_\bullet(S) \to \H^{\pf}_\bullet(S), \]
    and put $S^{E} \coloneqq \Sigma^{E}(1_S)$.
\end{Not}

\begin{Lem}
\label{Lem:thom_stabl_htpy_inv}
    The morphism $p^* \colon \SH^\pf(S) \to \SH^{\pf}(E)$ is fully faithful, and the functor $\Sigma^E$ is an equivalence.
\end{Lem}

\begin{proof}
    Since $\Phi$ commutes with $(-)_\sharp$, the diagram
    \begin{center}
        \begin{tikzcd}
            \SH(E_0)  & \SH(S) \arrow[l,swap, "p_0^*"] \\
            \SH^{\pf}(E)  \arrow[u, "\Psi_{E_0}"] & \SHpf(S) \arrow[u,swap, "\Psi_{S}"] \arrow[l,swap, "p^*"]
        \end{tikzcd}
    \end{center}
    commutes. Whence the first claim follows since $\Psi_S, \Psi_{E_0}$ are fully faithful by Proposition \ref{Prop:Psi_ff}.

    The diagram 
    \begin{center}
        \begin{tikzcd}
            \SH(S) \arrow[d, "\Phi_S"] \arrow[r, "{s_0}_*"] & \SH(E_0) \arrow[d, "\Phi_{E_0}"] \arrow[r, "{p_0}_\sharp"] & \SH(S) \arrow[d, "\Phi_S"] \\
            \SH^{\pf}(S) \arrow[r, "s_*"] & \SH^{\pf}(E) \arrow[r, "p_\sharp"] & \SH^{\pf}(S)
        \end{tikzcd}
    \end{center}
    commutes by Lemma~\ref{Lem:Phi_commutes_i}. Since $\Phi_S$ is a localization by Proposition \ref{Prop:Psi_ff} the second claim follows since ${p_0}_\sharp \circ {s_0}_*$ is an equivalence.
\end{proof}

\subsection{$\SH^{\pf}$ as coefficient system}

\begin{Cor}
\label{Cor:SH_motivic_coefficient_systems}
    The functor
    \[ \SH^\pf((-)^{\pf}) \colon \SchFp^\op \to \PrLotimesst \]
    is a coefficient system on $\SchFp$. In fact it is the initial coefficient system which inverts universal homeomorphisms.
\end{Cor}

\begin{proof}
     The functor $\SH^{\pf}((-)^\pf)$ satisfies homotopy invariance and Thom stability by Lemma~\ref{Lem:thom_stabl_htpy_inv}. Initiality will follow by a similar argument as for $\SH$, cf.\ \cite[Rem.~2.14]{KhanVoevodsky}, using Proposition~\ref{Prop:cdh_to_idh}. It remains to show localization.

    So let $i_0 \colon Z_0 \to S_0$ be a closed immersion with open complement $j_0 \colon U_0 \to S_0$, and perfections $i \colon Z \to S$ and $j \colon U \to S$. Consider the commutative diagram
    \begin{center}
        \begin{tikzcd}
            \SH(Z) \arrow[r, "i_*"] & \SH(S) \arrow[r, "j^*"] & \SH(U) \\
            \SH^{\pf}(Z) \arrow[u, "\Psi_Z"] \arrow[r, "i_*"] & \SH^{\pf}(S) \arrow[u, "\Psi_S"] \arrow[r, "j^*"] & \SH^{\pf}(U). \arrow[u, "\Psi_U"]
        \end{tikzcd}
    \end{center}
    Let $E \in \SH^{\pf}(S)$ such that $j^*E \simeq 0$ be given. Then
    \[ j^* \Psi_S E \simeq \Psi_U j^*E \simeq 0, \]
    hence there is some $F \in \SH(Z)$ such that $i_*F \simeq \Psi_S E$, and it holds
    \[ i_* \Phi_Z F \simeq \Phi_S i_* F \simeq \Phi_S \Psi_S E \simeq E, \]
    by Lemma~\ref{Lem:Phi_commutes_i} and Proposition~\ref{Prop:Psi_ff}, whence the claim.
\end{proof}

\section{Comparison results}
\label{Sec:comparison}
Let $S_0$ be an $\Fp$-scheme with perfection morphism $\rho \colon S \to S_0$. In this section, we shall see the equivalences in the following commutative diagram:
\begin{center}
    \begin{tikzcd}
        \SH(S_0) \arrow[r] \arrow[d] 
        & \SH_\idh(S_0) \arrow[r,"\simeq","\alpha"'] \arrow[d,"\simeq"]
        & \SH^\pf(S) \arrow[d,"\simeq","\gamma"']
        \\
        \SH(S_0)[1/p] \arrow[r,"\simeq","\beta"']
        & \SH_\idh(S_0)[1/p] \arrow[r,"\simeq","\alpha'"']
        & \SHpf(S)[1/p].
    \end{tikzcd}
\end{center}
Invertibility of $\beta$ (Proposition~\ref{Prop:SH=SHidh-away-from-p}) combines Khan's argument for showing that $\SH(S_0)\simeq\SH_\cdh(S_0)$ \cite{Khan-SH_cdh} with Elmanto--Khan's result that universal homeomorphisms induce equivalences in $\SH(-)[1/p]$ \cite[Thm.~2.1.1]{Elmanto-Khan-perfection}. 
Invertibility of $\alpha$ hence of $\alpha'$ (Proposition~\ref{Prop:SHidh=SHpf}) uses the comparison of the $\idh$-topos and the $\cdh$-topos (Proposition~\ref{Prop:cdh_to_idh}). 
Finally, we prove invertibility of $\gamma$ (Proposition~\ref{Prop:SHpf-p-invertible}).
Moreover:
\begin{itemize}
    \item The map $\Phi_{S} \colon \SH(S) \to \SHpf(S)$ is not invertible in general: take $S = \Spec \F_p$, and note that $\K_1(\G_m) \simeq \K_1(\F_p)\oplus\K_0(\F_p)\simeq \F_p^\times\oplus\Z$ is not $p$-divisible, meanwhile the higher K-groups of perfect schemes are $\Z[1/p]$-modules \cite[Thm.~1.1]{AntieauPerfectoid}.
    \item The map $\rho^*\colon \SH(S_0) \to \SH(S)$ is not invertible in general: take $S_0 = \G_m$ and consider the same computation as in the previous point.
\end{itemize}

\begin{Prop} \label{Prop:SHpf-p-invertible}
    For any object $X\in\SHpf(S)$, the multiplication-by-$p$ map $X\to[\cdot p]X$ is an equivalence. Hence the inclusion functor 
    \[  \SHpf(S)[1/p] \to \SHpf(S) \]
    is an equivalence which is inverse to its left-adjoint $\gamma$.
\end{Prop}
\begin{proof}
    First, let $S=\Spec(\F_p)$ and $X=\bar\G_m$. Recall that there are ring isomorphisms
    \begin{align*}
     \pi_0\Hom_{\SH(\F_p)}(\G_m,\G_m) \simeq \K^\mathrm{MW}_0(\F_p)  \simeq \mathrm{GW}(\F_p)
     \end{align*}
    by \cite[Cor.~C.9]{AyoubRealisationEtale} and  \cite[Lem.~3.10]{morel-A1-book}.\footnote{Note that $\K^\mathrm{MW}_0(F) \simeq \mathrm{GW}(F)$ for every field $F$ in \emph{any} characteristic, when $\mathrm{GW}(F)$ is defined as the ring of isomorphism classes of non-degenerate symmetric bilinear forms, cf.\ \cite[p.~52f.]{morel-A1-book}.}
    Under these isomorphisms, the Frobenius $\Frob_{\G_m}$ corresponds to the element $p_\epsilon \coloneqq \sum_{i=1}^p \langle (-1)^{i-1} \rangle$ in $\K^\mathrm{MW}_0(\F_p)$ \cite[Proof of Cor.~C.5, p.~142]{AyoubRealisationEtale}, cf.\ also \cite[Ex.~4.5]{richarz-scholbach-frobenius}.
    Note that $p = p_\epsilon + \tfrac{p-1}{2}\cdot h$ for the hyperbolic form $h:=\langle 1\rangle - \langle -1\rangle$ satisfying $h^2=0$, hence that $p-p_\epsilon$ is nilpotent.
    Alternatively one can argue that the element $p-p_\epsilon$ is nilpotent as it lies in the kernel of the ring homomorphism $\mathrm{rk} \colon \mathrm{GW}(\F_p) \to \Z$ which is cyclic of order at most 2  \cite[Ex.~2.1.5]{deglise--milnorwitt}.    
    Consequently, for every ring $R$ and every ring homomorphism $f\colon \mathrm{GW}(\F_p) \to R$ it holds that $f(p)\in R^\times$ if and only if $f(p_\epsilon)\in R^\times$.
    
    The functor $\Phi_S$ induces a map
    \[ \Hom_{\SH(S)}(\G_m,\G_m) \to \Hom_{\SHpf(S)}(\bar\G_m,\bar\G_m)\]
    sending $p_\epsilon=\Frob_{\G_m}$ to $\Frob_{\bar\G_m}$. Since the latter is an equivalence, it follows that the multiplication map $\bar\G_m\to[\cdot p]\bar\G_m$ is an equivalence. Since $\bar\G_m$ is $\otimes$-invertible in $\SHpf(\F_p)$, also the endomorphism $p \in \End_{\SH^\pf}(1)$ is invertible, thus the claim for $S=\Spec(\F_p)$ follows.
    
    For an arbitrary $S\to\Spec(\F_p)$ the claim follows by functoriality.
\end{proof}

Recall that $\uH(S_0)$ (resp.\ $\uSH(S_0)$) is the category of $\SchfpSn
$-fibered motivic spaces (resp.\ spectra), and that $\uHpf(S)$ (resp.\ $\uSHpf(S)$) is the perfect analogue, see \S \ref{Subsec:fibred-spaces-spectra}. 
Following Khan \cite[\S 3]{Khan-SH_cdh}, for $\tau \in \{\cdh,\idh\}$ write $\uH_\tau(S_0)$ for the full subcategory of $\uH(S_0)$ spanned by the $\tau$-local objects, and write $\uSH_\tau(S_0)$ for the full subcategory of $\uSH(S_0)$ spanned by $\T$-spectra of $\tau$-local objects. 
Recall that the inclusion $\iota_!  \colon \SH(S) \inj \SHu(S)$ factors as
\[ \SH(S_0) \xrightarrow{\iota_!} \uSH_{\cdh}(S_0) \hookrightarrow \uSH(S_0), \]
and that the essential image is written as $\SH_{\cdh}(S_0)$ \cite[Thm.~9]{Khan-SH_cdh}. Consequently, $\SH_\cdh(S_0)$ is also the essential image of the composition
\[\SH(S_0) \xrightarrow{\iota_!} \uSH(S_0) \xrightarrow{\L_\cdh} \uSH_\cdh(S_0).\]

Write $\uHpf_{\pcdh}(S)$ for the full subcategory of $\uHpf(S)$ spanned by the $\pcdh$-local objects, and $\uSHpf_{\pcdh}(S)$ for the full subcategory of $\uSHpf(S)$ spanned by $\barT$-spectra of $\pcdh$-local objects. 

\begin{Lem}
\label{Lem:iota_SHpcdh}
    The functor $\iota_! \colon \SH^\pf(S) \to \uSHpf(S)$ lands in $\uSHpf_{\pcdh}(S)$.
\end{Lem}

\begin{proof}
    Note that $\iota_!$ commutes with the localization $\Phi$ in the obvious way. Hence, by \cite[Thm.~9]{Khan-SH_cdh} it suffices to show that 
    \[ \uPhi \colon \uSH(S) \to \uSH^\pf(S) \]
    sends $\uSH_\cdh(S)$ to $\uSH^\pf_{\pcdh}(S)$. Since $\uPhi$ is the stabilization of $\uphi$, this follows from Corollary~\ref{Cor:cdh_vs_pcdh}.
\end{proof}

\begin{Def}\label{Def:SHidh}
As in the ordinary case, write $\SHpf_{\pcdh}(S) \subset \uSHpf(S)$ for the essential image of $\iota_!$. Likewise, define $\SH_\idh(S_0)$ as the essential image of the composition $\SH(S_0) \xrightarrow{\iota_!} \uSH(S_0) \xrightarrow{\L_\idh} \uSH_\idh(S_0)$.
\end{Def}

\begin{Prop} \label{Prop:SHidh=SHpf}
The perfection $\rho \colon \SchfpSn \to \PerfpfpS$ induces equivalences
\[ \SH_\idh(S_0) \simeq \SH^\pf_\pcdh(S) \simeq \SHpf(S). \]
\end{Prop}
\begin{proof}
The equivalence $\SH^\pf_\pcdh(S) \simeq \SH^\pf(S)$ is true by definition, via Lemma~\ref{Lem:iota_SHpcdh}.
For the remaining equivalence, note that the perfection functor $\rho \colon \SchfpSn \to \PerfpfpS$ induces an equivalence $\rho_! \colon \Sh_\idh(\SchfpSn) \xrightarrow{\sim} \Sh_\pcdh(\PerfpfpS)$ by Corollary~\ref{Cor:idh=pcdh}.
Now $\rho_!(\A^1_{S_0}) \simeq \barA_S$, hence we get an induced equivalence $\underline{\H}_\idh(S_0) \xrightarrow{\sim} \underline{\H}^\pf_\pcdh(S)$.
Similarly, since $\rho_!(\Lmot \P^1_{S_0}) \simeq \Lmot \overline{\P}^1_S$, we get an induced equivalence $\underline{\SH}_\idh(S_0) \xrightarrow{\sim}\underline{\SH}^\pf_\pcdh(S)$.

By construction we have a commutative diagram
\begin{center}
    \begin{tikzcd}
        \SH(S_0) \arrow[d, "\Phi"] \arrow[r, "\iota_!"] & 
        \uSH(S_0) \arrow[r, "\L_\idh"] \arrow[d, "\uPhi"] & 
        \uSH_\idh(S_0) \arrow[d, "\rho_!","\simeq"'] \\
        \SHpf(S) \arrow[r, "\iota_!"] &
        \uSHpf(S) \arrow[r, "\L_{\pcdh}"] &
        \uSH^\pf_\pcdh(S).
    \end{tikzcd}
\end{center}
Now $\SH_\idh(S_0)$ is the essential image of $\L_\idh  \iota_!$ by definition, and $\SHpf_\pcdh(S)$  is the essential image of $\L_\pcdh \iota_!$ by Lemma~\ref{Lem:iota_SHpcdh}. Whence the claim, since $\rho_!$ is invertible and $\Phi$ is a localization.
\end{proof}

\begin{Prop} \label{Prop:SH=SHidh-away-from-p}
There is a canonical equivalence 
\[ \SH(S_0)[1/p] \xrightarrow{\sim} \SH_\idh(S_0)[1/p]. \]
\end{Prop}
\begin{proof}
Consider the functor $\L_\idh\circ\iota_! \colon \SH(S_0) \to \SH_\idh(S_0)$ from Definition~\ref{Def:SHidh}. We claim that its localization $\SH(S_0)[1/p] \to \SH_\idh(S_0)[1/p]$ is an equivalence.
Given a universal homeomorphism $f\colon Y\to X$ in $\SchfpSn$, the  functor
\[ f^* \colon \SH(X)[1/p] \to \SH(Y)[1/p] \]
is an equivalence \cite[Cor.~2.1.5]{Elmanto-Khan-perfection}. In particular, for every $G\in\SH(X)[1/p]$ the map $G \to f_*f^*G$ is an equivalence. Denoting $p\colon X\to S$ and $q\colon Y\to S$ the structure morphisms (so that $q=pf$), we get for every $F\in\SH(S_0)[1/p]$ that 
\begin{align*}
         \Gamma(X,\iota_!F) &\simeq \Hom_{\SH(S_0)[1/p]}(1_S,p_*p^*F) \\
         &\simeq \Hom_{\SH(S_0)[1/p]}(1_S,q_*q^*F) \simeq \Gamma(Y,\iota_!F)
\end{align*}
using \cite[Prop.~5]{Khan-SH_cdh}. Thus $\iota_!F$ lies in $\SHu_\idh(S_0)[1/p]$ so that the functor $\SH(S_0)[1/p] \to \SH_\idh(S_0)[1/p]$ is fully faithful. As it is essentially surjective by construction, it is an equivalence.
\end{proof}

\begin{Thm} \label{Thm:comparison-results-all}
    We have a sequence of natural equivalences
    \[ \SH^\pf(S)[1/p] \simeq \SH^\pf(S) \simeq \SH_\idh(S_0) \simeq  \SH_\idh(S_0)[1/p] \simeq \SH(S_0)[1/p]. \]
\end{Thm}
\begin{proof}
    This follows by combining Proposition~\ref{Prop:SHpf-p-invertible}, Proposition~\ref{Prop:SHidh=SHpf}, and Proposition~\ref{Prop:SH=SHidh-away-from-p}.
\end{proof}

\begin{Rem}
    Strictly speaking, one does not need the full force of $\SHpf$ as coefficient system in order to carry out the comparison $\SHpf \simeq \SH[1/p]$: it is sufficient to have Proposition~\ref{Prop:Phi_comm_susp} and Proposition~\ref{Prop:Psi_ff}. From there, Theorem~\ref{Thm:comparison-results-all} implies Corollary~\ref{Cor:SH_motivic_coefficient_systems}, since we already know that $\SH[1/p]$ is a coefficient system. However, we expect some of the arguments that went into Corollary~\ref{Cor:SH_motivic_coefficient_systems} to be of interest in their own right, for example in constructing six functor formalisms that invert universal homeomorphisms beyond the characteristic $p$ case.   
\end{Rem}

\section{K-theory}
\label{Sec:K-theory}
We finish this article by observing representability of algebraic K-theory in the unstable perfect motivic homotopy category. Furthermore, we identify its representing object with $\Z\times\BGL$, as in the classical case \cite[Thm.~3.13]{MorelVoevodskyA1}.\footnote{Be aware that the proof in loc.\ cit.\ relies on Proposition~1.9 in loc.\ cit.\ which is false. This has been fixed by Schlichting--Tripathi, see \cite[Rem.~2, p.~1162]{schlichting-tripathi}.}
For any scheme $X$, write $\K(X)$ for the algebraic $\K$-theory of $X$, obtained as the (nonconnective) $\K$-theory spectrum of the $\infty$-category $\Perf(S)$ of perfect complexes on $S$, cf.\ \cite{thomason-trobaugh,bgt13}.

We still let $S_0$ be an $\Fp$-scheme with perfection $S$.
For any smooth $S$-scheme $X$ the canonical map $\K(X) \to \KH(X)$ is an equivalence; this can be checked locally on affine open subsets where it is known by a result of Antieau--Mathew--Morrow \cite[Prop.~5.1]{AntieauPerfectoid}.\footnote{However, algebraic K-theory of perfect schemes may be non-connective \cite[\S 7.2]{LowitEquivariant}.}
In particular, algebraic K-theory is $\A^1$-invariant on perfect $\F_p$-schemes. From this, we also can deduce $\barA$-invariance:

\begin{Prop} \label{Prop:K-connective-lies-in-Hpf}
    As a spectrum-valued presheaf on $\PerfpfpS$, algebraic K-theory is an $\barA$-invariant perfect Nisnevich sheaf.
    In particular, connective algebraic K-theory is canonically an object in $\H^\pf(S)$.
\end{Prop}
\begin{proof}
By definition, a perfect Nisnevich square is just a Nisnevich square which happens to consist of perfect schemes. Hence algebraic K-theory is a perfect Nisnevich sheaf as it is a Nisnevich sheaf.

Now let $X\in\PerfpfpS$. By $\A^1$-invariance of K-theory for perfect schemes, the map
\begin{center} \begin{tikzcd}
    \K(X) \arrow[d] & \arrow[l,"\simeq",swap] \colim\bigl( \K(X) \arrow[d,"\simeq"] \arrow[r,"\phi^*"] & \K(X) \arrow[d,"\simeq"] \arrow[r,"\phi^*"] & \ldots \bigr) \\
    \K(\barA_X) & \arrow[l,"\simeq",swap] \colim\bigl( \K(\A^1_X) \arrow[r,"\phi^*"] & \K(\A^1_X) \arrow[r,"\phi^*"] & \ldots \bigr) 
\end{tikzcd} \end{center}
is an equivalence, where $\phi^*$ denotes the map induced by the Frobenius and we have used  continuity of K-theory for the equivalences in the horizontal direction \cite[Thm.~7.2]{thomason-trobaugh}.
\end{proof}

Write $\K_{\geq 0}^\pf$ for the representing object of connective algebraic K-theory in $\H^\pf(S)$, and likewise for $\K_{\geq 0} \in \H(S)$. Recall that $\varphi^\PPP \colon \PPP(\Sm_{S}) \to \PPP(\Sm_S^\pf)$ is left Kan extension along the perfection functor.

\begin{Lem}
\label{Lem:varphi_K}
    There are canonical equivalences $\varphi^\PPP_{S}(\K_{\geq 0}) \simeq \varphi_{S}(\K_{\geq 0}) \simeq \K_{\geq 0}^\pf$.
\end{Lem}

\begin{proof}
    Take $X \in \Sm_S^{\pf '}$ with smooth model $X_0$. It holds
    \[ \varphi^\PPP(\K_{\geq 0})(X) \simeq \colim_{X_\alpha \in X/\Sm_S}\K_{\geq 0}(X_\alpha). \]
    Let $\FFF \subset X/\Sm_S$ be the full subcategory spanned by factorizations of the form $X \to X_n \to S$, where $X_n \to S$ is induced by the $n$-fold Frobenius on $X_0$ and $X \to X_n$ is the perfection. By a similar argument as in the proof of Lemma~\ref{Lem:phi_iota}, the inclusion $\FFF \to X/\Sm_S$ is initial, so that
    \[ \varphi^\PPP(\K_{\geq 0})(X) \simeq \colim_{X_n \in \FFF}\K_{\geq 0}(X_n) \simeq \K_{\geq 0}(X) \]
    by continuity of K-theory. Now $X \mapsto \K_{\geq 0}(X)$ is already an $\barA$-invariant Nisnevich sheaf by Proposition~\ref{Prop:K-connective-lies-in-Hpf}, whence the claim.
\end{proof}

We write $\BGL \coloneqq \colim_n \BGL_n$ both in $\H$ and in $\Hpf$, where $\BGL_n$ is the classifying space of $\GL_n$ (and we suppress the (perfect) motivic localization from notation).

\begin{Lem}
\label{Lem:phi_BGL}
    It holds that $\varphi^\PPP(\Z \times \BGL) \simeq \Z \times \BGL$.
\end{Lem}
\begin{proof}
Since  $\varphi^\PPP$ is symmetric monoidal and colimit-preserving, we get that
\[ \varphi^\PPP(\Z\times \BGL) \simeq \Z \times \colim_n \varphi^\PPP(\BGL_n).\]
Since $\B G \simeq \colim_{[n]\in\Delta^\op} G^{n-1}$ for any group object $G$, it suffices to show that $\varphi^\PPP(\GL_n) \simeq \GL_n$. The latter follows by the fact that $\GL_n$ is representable, together with Lemma~\ref{Lem:phiP_yoneda} and the universal property of $\GL_n$. 
\end{proof}

\begin{Cor} \label{Cor:K=ZxBGL}
     We have equivalences    
    \[ \K^\pf_{\geq 0} \simeq \varphi(\K_{\geq 0}) \simeq  \varphi^\PPP(\K_{\geq 0}) \simeq \varphi^\PPP(\Z \times \BGL) \simeq \Z \times \BGL \]
    in $\H^\pf(S)$. 
\end{Cor}
\begin{proof}
    Since $\K_{\geq 0} \simeq \Z \times \BGL$ holds in $\H(S)$, the claim follows from Lemma~\ref{Lem:varphi_K} and Lemma~\ref{Lem:phi_BGL}.
\end{proof}

\begin{Rem}
     Observe that algebraic $\K$-theory is not representable in $\SH^\pf$ by Proposition~\ref{Prop:SHpf-p-invertible}, since $\K_0(X)$ is not a $\Z[1/p]$-module in general, even for perfect schemes $X$, e.g., $\K_0(\F_p)\simeq\Z$. On the other hand, we have seen that connective algebraic K-theory is representable in $\H^\pf$ in Proposition~\ref{Prop:K-connective-lies-in-Hpf}. 
     
     In the ordinary setting, the Bass Fundamental Theorem is at the heart of concluding the representability of algebraic K-theory in $\SH$ from the corresponding statement in $\H$. In contrast, the perfect analogue of the Bass Fundamental Theorem does not hold true, that is, for $n\geq 1$ there is not necessarily an exact sequence
    \[ 0 \to \K_n(X) \to \K_n(\barA_X)\oplus \K_n(\barA[-1]_X) \to \K_n(\bar{\G}_{m,X}) \to \K_{n-1}(X) \to 0. \]
    Indeed, setting $X=\Spec(\F_p)$ and $n=1$ the Bass Fundamental Theorem yields an exact sequence
    \[ 0 \to \K_1(\F_p) \to \K_1(\F_p[t])\oplus \K_1(\F_p[t\inv]) \to \K_1(\F_p[t,t\inv]) \to \K_0(\F_p) \to 0 \]
    that identifies with the exact sequence
    \[ 0 \to \F_p^\times \to[\Delta] \F_p^\times \oplus \F_p^\times \to[(\pm,0)] \F_p^\times \oplus \Z \to[\mathrm{pr}_2] \Z \to 0, \]
   where $(\pm,0)$ sends $(x,y)$ to $(xy^{-1},0)$.  
    Applying the colimit along the maps induced by the Frobenius on $\G_{m,\F_p}$ we get that $\K_1(\bar{\G}_{m}) \cong \F_p^\times \oplus \Z[1/p]$ which cannot map surjectively to $\K_0(\F_p)$.
\end{Rem}

\bibliographystyle{dary}
\bibliography{refs}	

\providecommand{\MR}{\relax\ifhmode\unskip\space\fi MR }
\providecommand{\MRhref}[2]{%
  \href{http://www.ams.org/mathscinet-getitem?mr=#1}{#2}
}
\providecommand{\href}[2]{#2}
\begin{thebibliography}{AMM22}

\bibitem[AI23]{annala-iwasa-motivic-spectra}
Toni Annala and Ryomei Iwasa, \emph{Motivic spectra and universality of
  {$K$}-theory}, 2023,
  \href{http://arXiv.org/abs/2204.03434v2}{\mbox{arXiv:2204.03434v2}}.

\bibitem[AMM22]{AntieauPerfectoid}
Benjamin Antieau, Akhil Mathew, and Matthew Morrow, \emph{The {K}-theory of
  perfectoid rings}, Doc. Math. \textbf{27} (2022), 1923--1952.

\bibitem[Ayo07]{AyoubThesis}
Joseph Ayoub, \emph{Les six op\'{e}rations de {G}rothendieck et le formalisme
  des cycles \'{e}vanescents dans le monde motivique. {(I, II)}},
  Ast\'{e}risque (2007), no.~314, 315.

\bibitem[Ayo14]{AyoubRealisationEtale}
Joseph Ayoub, \emph{La r\'ealisation \'etale et les op\'erations de
  {G}rothendieck}, Ann. Sci. \'Ec. Norm. Sup\'er. (4) \textbf{47} (2014),
  no.~1, 1--145.

\bibitem[Bar10]{BarwickTopological}
C.~Barwick, \emph{Topological rigidification of schemes}, 2010,
  \href{http://arXiv.org/abs/1012.1889v1}{\mbox{arXiv:1012.1889v1}}.

\bibitem[BD17]{BondarkoDeglise2017}
Mikhail Bondarko and Fr\'ed\'eric D\'eglise, \emph{Dimensional homotopy
  t-structures in motivic homotopy theory}, Adv. Math. \textbf{311} (2017),
  91--189.

\bibitem[BE25]{BachmannNotes}
Tom Bachmann and Elden Elmanto, \emph{Notes on motivic infinite loop space
  theory}, Motivic geometry, Open Book Ser., vol.~6, Math. Sci. Publ.,
  Berkeley, CA, 2025, pp.~1--62.

\bibitem[BGH20]{BarwickExodromy}
Clark Barwick, Saul Glasman, and Peter Haine, \emph{Exodromy}, 2020,
  \href{http://arXiv.org/abs/1807.03281}{\mbox{arXiv:1807.03281}}.

\bibitem[BGT13]{bgt13}
Andrew~J. Blumberg, David Gepner, and Gon\c{c}alo Tabuada, \emph{A universal
  characterization of higher algebraic {$K$}-theory}, Geom. Topol. \textbf{17}
  (2013), no.~2, 733--838.

\bibitem[BH21]{BachmannNorms}
Tom Bachmann and Marc Hoyois, \emph{Norms in motivic homotopy theory},
  Ast\'erisque (2021), no.~425, ix+207.

\bibitem[BS17]{BhattProjectivity}
Bhargav Bhatt and Peter Scholze, \emph{Projectivity of the {W}itt vector affine
  {G}rassmannian}, Invent. Math. \textbf{209} (2017), no.~2, 329--423.

\bibitem[CD19]{CisinskiDeglise}
Denis-Charles Cisinski and Fr\'{e}d\'{e}ric D\'{e}glise, \emph{Triangulated
  categories of mixed motives}, Springer Monographs in Mathematics, Springer,
  Cham, 2019.

\bibitem[CHW24]{CarlsonReconstructionschemesetaletopoi}
Magnus Carlson, Peter~J. Haine, and Sebastian Wolf, \emph{Reconstruction of
  schemes from their \'{e}tale topoi}, 2024,
  \href{http://arXiv.org/abs/2407.19920}{\mbox{arXiv:2407.19920}}.

\bibitem[CM21]{clausen-mathew--hyperdescent}
Dustin Clausen and Akhil Mathew, \emph{Hyperdescent and \'etale {$K$}-theory},
  Invent. Math. \textbf{225} (2021), no.~3, 981--1076.

\bibitem[DHW25]{DHW1}
Christian Dahlhausen, Jeroen Hekking, and Storm Wolters, \emph{Duality for
  {KGL}-modules in motivic homotopy theory}, 2025,
  \href{http://arXiv.org/abs/2508.00064}{\mbox{arXiv:2508.00064}}.

\bibitem[Dé23]{deglise--milnorwitt}
Frédéric Déglise, \emph{Notes on {M}ilnor-{W}itt {K}-theory}, 2023,
  \href{http://arXiv.org/abs/2305.18609}{\mbox{arXiv:2305.18609}}.

\bibitem[EGA4]{EGA4}
Jean Dieudonn{\'e} and Alexander Grothendieck, \emph{\'{E}l\'ements de
  g\'eom\'etrie alg\'ebrique}, Inst. Hautes \'Etudes Sci. Publ. Math.
  \textbf{4} (1961--1967).

\bibitem[EHIK21]{ehik--milnor-excision}
Elden Elmanto, Marc Hoyois, Ryomei Iwasa, and Shane Kelly, \emph{Cdh descent,
  cdarc descent, and {M}ilnor excision}, Math. Ann. \textbf{379} (2021),
  no.~3-4, 1011--1045.

\bibitem[EK20]{Elmanto-Khan-perfection}
Elden Elmanto and Adeel~A. Khan, \emph{Perfection in motivic homotopy theory},
  Proc. Lond. Math. Soc. (3) \textbf{120} (2020), no.~1, 28--38.

\bibitem[Gal25]{GallauerSix}
Martin Gallauer, \emph{An introduction to six-functor formalisms}, Motivic
  geometry, Open Book Ser., vol.~6, Math. Sci. Publ., Berkeley, CA, 2025,
  pp.~63--106.

\bibitem[GK15]{GabberPoints}
Ofer Gabber and Shane Kelly, \emph{Points in algebraic geometry}, J. Pure Appl.
  Algebra \textbf{219} (2015), no.~10, 4667--4680.

\bibitem[GL01]{goodwillie-lichtenbaum}
Thomas~G. Goodwillie and Stephen Lichtenbaum, \emph{A cohomological bound for
  the {$h$}-topology}, Amer. J. Math. \textbf{123} (2001), no.~3, 425--443.

\bibitem[Hov01]{hovey-symmetric-spectra}
Mark Hovey, \emph{Spectra and symmetric spectra in general model categories},
  J. Pure Appl. Algebra \textbf{165} (2001), no.~1, 63--127.

\bibitem[Hoy14]{hoyois-quadratic}
Marc Hoyois, \emph{A quadratic refinement of the
  {G}rothendieck-{L}efschetz-{V}erdier trace formula}, Algebr. Geom. Topol.
  \textbf{14} (2014), no.~6, 3603--3658.

\bibitem[Hoy17]{HoyoisSix}
Marc Hoyois, \emph{The six operations in equivariant motivic homotopy theory},
  Adv. Math. \textbf{305} (2017), 197--279.

\bibitem[Kha19]{KhanMorelVoevodsky}
Adeel~A. Khan, \emph{The {M}orel-{V}oevodsky localization theorem in spectral
  algebraic geometry}, Geom. Topol. \textbf{23} (2019), no.~7, 3647--3685.

\bibitem[Kha21]{KhanVoevodsky}
Adeel~A. Khan, \emph{Voevodsky's criterion for constructible categories of
  coefficients}, Date: 2021-01-25 (some revisions on 2023-03-17)
  \url{https://www.preschema.com/papers/six.pdf}, 2021.

\bibitem[Kha24]{Khan-SH_cdh}
Adeel~A. Khan, \emph{The cdh-local motivic homotopy category}, J. Pure Appl.
  Algebra \textbf{228} (2024), no.~5, Paper No. 107562, 6.

\bibitem[Lur03]{LurieOnInftyTopoi}
Jacob Lurie, \emph{On infinity topoi}, 2003,
  \href{http://arXiv.org/abs/0306109v2}{\mbox{arXiv:0306109v2}}.

\bibitem[Lur09]{LurieHTT}
Jacob Lurie, \emph{Higher topos theory}, Annals of Mathematics Studies, vol.
  170, Princeton University Press, Princeton, NJ, 2009.

\bibitem[Lur17]{LurieHA}
Jacob Lurie, \emph{Higher {A}lgebra}, (version dated September 18, 2017)
  \url{https://www.math.ias.edu/~lurie/papers/HA.pdf}, 2017.

\bibitem[LZ24]{liu-zheng}
Yifeng Liu and Weizhe Zheng, \emph{Enhanced six operations and base change
  theorem for higher {A}rtin stacks}, 2024,
  \href{http://arXiv.org/abs/1211.5948}{\mbox{arXiv:1211.5948}}.

\bibitem[Lö24]{LowitEquivariant}
Jakub Löwit, \emph{Equivariant {$K$}-theory, affine {G}rassmannian and
  perfection}, 2024,
  \href{http://arXiv.org/abs/2409.18925v2}{\mbox{arXiv:2409.18925v2}}.

\bibitem[Mor12]{morel-A1-book}
Fabien Morel, \emph{{$\Bbb A^1$}-algebraic topology over a field}, Lecture
  Notes in Mathematics, vol. 2052, Springer, Heidelberg, 2012.

\bibitem[MV99]{MorelVoevodskyA1}
Fabien Morel and Vladimir Voevodsky, \emph{$ {A}^1$-homotopy theory of
  schemes}, Publications math\'ematiques de l'IH\'ES \textbf{90} (1999),
  45--143.

\bibitem[NS17]{NikolausPresentably}
Thomas Nikolaus and Steffen Sagave, \emph{Presentably symmetric monoidal
  {$\infty$}-categories are represented by symmetric monoidal model
  categories}, Algebraic \& Geometric Topology \textbf{17} (2017), no.~5,
  3189--3212.

\bibitem[Rob14]{robalo-these}
Marco Robalo, \emph{Th\'eorie homotopique motivique des espaces
  non-commutatifs}, Ph.D. thesis, PhD-thesis, University of Montpellier, 2014.

\bibitem[Rob15]{robalo-bridge}
Marco Robalo, \emph{{$K$}-theory and the bridge from motives to noncommutative
  motives}, Adv. Math. \textbf{269} (2015), 399--550.

\bibitem[RS21]{richarz-scholbach}
Timo Richarz and Jakob Scholbach, \emph{Tate motives on {W}itt vector affine
  flag varieties}, Selecta Math. (N.S.) \textbf{27} (2021), no.~3, Paper No.
  44, 34.

\bibitem[RS25]{richarz-scholbach-frobenius}
Timo Richarz and Jakob Scholbach, \emph{Frobenius rigidity in {$\Bbb
  A^1$}-homotopy theory}, Doc. Math. \textbf{30} (2025), no.~1, 219--244.

\bibitem[Ryd10]{RydhSubmersions}
David Rydh, \emph{Submersions and effective descent of \'etale morphisms},
  Bull. Soc. Math. France \textbf{138} (2010), no.~2, 181--230.

\bibitem[ST15]{schlichting-tripathi}
Marco Schlichting and Girja~S. Tripathi, \emph{Geometric models for higher
  {G}rothendieck-{W}itt groups in {$\Bbb{A}^1$}-homotopy theory}, Math. Ann.
  \textbf{362} (2015), no.~3-4, 1143--1167.

\bibitem[{Sta}]{stacks-project}
The {Stacks Project Authors}, \emph{\textit{Stacks Project}},
  \url{http://stacks.math.columbia.edu}.

\bibitem[TT90]{thomason-trobaugh}
Robert~W. Thomason and Thomas Trobaugh, \emph{Higher algebraic {$K$}-theory of
  schemes and of derived categories}, The {G}rothendieck {F}estschrift, {V}ol.\
  {III}, Progr. Math., vol.~88, Birkh\"auser Boston, Boston, MA, 1990,
  pp.~247--435.

\bibitem[Voe96]{VoevodskyHomology}
Vladimir Voevodsky, \emph{Homology of schemes}, Selecta Math. (N.S.) \textbf{2}
  (1996), no.~1, 111--153.

\bibitem[Voe98]{voevodsky-icm-1998}
Vladimir Voevodsky, \emph{{${\bf A}^1$}-{H}omotopy {T}heory}, Proceedings of
  the {I}nternational {C}ongress of {M}athematicians, {V}ol. {I} ({B}erlin,
  1998), 1998, pp.~579--604.

\bibitem[Zhu17]{ZhuAffine}
Xinwen Zhu, \emph{Affine {G}rassmannians and the geometric {S}atake in mixed
  characteristic}, Ann. of Math. (2) \textbf{185} (2017), no.~2, 403--492.

\end{thebibliography}
\end{document}